\documentclass{amsart}
\usepackage[utf8]{inputenc}
\usepackage{amsmath, amsfonts, amsthm, amssymb,mathrsfs,color,enumitem}
\usepackage[margin=1.5in]{geometry}
\usepackage{graphicx}

\usepackage[dvipsnames]{xcolor} 

\usepackage[hidelinks]{hyperref}
\hypersetup{linktocpage,bookmarksopen = true, bookmarksopenlevel = 1, colorlinks=true,
linkcolor = Bittersweet, citecolor = RoyalBlue , urlcolor = gray, pdfstartview = FitR}
\usepackage{caption, subcaption} 
\usepackage{soul}
\usepackage{array}
\newcommand{\PreserveBackslash}[1]{\let\temp=\\#1\let\\=\temp}
\newcolumntype{C}[1]{>{\PreserveBackslash\centering}p{#1}}

\newcommand{\bs}[1]{\textcolor{Bittersweet}{#1}}

\newtheorem*{theorem*}{Theorem}
\newtheorem{thm}{Theorem}[subsection] 
\newtheorem{theorem}[thm]{Theorem} 
 
\newtheorem{conjecture}[thm]{Conjecture}
\newtheorem*{conjecture*}{Conjecture}
\newtheorem{proposition}[thm]{Proposition}
\newtheorem{lemma}[thm]{Lemma}

\theoremstyle{definition}
\newtheorem{example}[thm]{Example}

\newtheorem{definition}[thm]{Definition}
\newtheorem{remark}[thm]{Remark}

\definecolor{myorange}{rgb}{0.9, 0.55, 0.3}
\definecolor{mygreen}{rgb}{0.35, 0.71, 0.0}
\definecolor{mybrown}{rgb}{0.63, 0.32, 0.18}

\newcommand{\ZZ}{\mathbb{Z}}

\newcommand{\ZZsub}{\mathbb{Z}_}
\newcommand{\Pth}{{\mathcal{P}^{\operatorname{th}}}}
\newcommand{\Ppr}{{\mathcal{P}^{\operatorname{pr}}}}
\newcommand{\Pfull}{{\mathcal{P}^{\operatorname{full}}}}
\newcommand{\Bdens}{{\mathcal{B}}}
\newcommand{\countingupperbound}{\leq}

\newcommand{\cirfull}{C_\Pfull(X)}
\newcommand{\cirth}{C_\Pth(X)}
\newcommand{\cirpr}{C_\Ppr(X)}
\newcommand{\curfull}{N_\Pfull(X)}
\newcommand{\curth}{N_\Pth(X)}

\title{Prime and thickened prime components in Apollonian circle packings}

\author[Friedlander]{Holley Friedlander}
\address{Dickinson College, Carlisle, Pennsylvania, USA}
\email{friedlah@dickinson.edu}
\urladdr{https://sites.google.com/view/hfriedlander}

\author[Fuchs]{Elena Fuchs}
\address{University of California Davis, Davis, California, USA}
\email{efuchs@math.ucdavis.edu}
\urladdr{https://www.math.ucdavis.edu/~efuchs/}

\author[Harris]{Piper Harris}
\email{drpiperh@liberatemath.org}

\author[Hsu]{Catherine Hsu}
\address{Swarthmore College, Swarthmore, Pennsylvania, USA}
\email{chsu2@swarthmore.edu}
\urladdr{https://chsu.domains.swarthmore.edu/}

\author[Rickards]{James Rickards}
\address{Saint Mary's University, Halifax, Nova Scotia, Canada}
\email{james.rickards@smu.ca}
\urladdr{https://jamesrickards-canada.github.io/}

\author[Sanden]{Katherine Sanden}
\email{katherine.sanden@gmail.com}

\author[Schindler]{Damaris Schindler}
\address{Göttingen University, Göttingen, Germany}
\email{damaris.schindler@mathematik.uni-goettingen.de}
\urladdr{https://sites.google.com/site/damarishomepage/}

\author[Stange]{Katherine E. Stange}
\address{University of Colorado Boulder, Boulder, Colorado, USA}
\email{kstange@math.colorado.edu}
\urladdr{https://math.katestange.net/}

\date{\today}
\thanks{This material is based upon work supported by the National Security Agency
under Grant No. H98230-19-1-0119, The Lyda Hill Foundation, The McGovern
Foundation, and Microsoft Research, while the authors were in residence at the
Mathematical Sciences Research Institute in Berkeley, California, during the
summer of 2019. 
Rickards and Stange were supported by NSF-CAREER CNS-1652238 (PI Stange).  Stange was supported by NSF DMS-2401580.
Schindler was supported by a NWO grant 016.Veni.173.016. Fuchs was supported by NSF Grant DMS-2154624.
This work utilized the Alpine high performance computing resource at the University of Colorado Boulder. Alpine is jointly funded by the University of Colorado Boulder, the University of Colorado Anschutz, and Colorado State University.}

\makeatletter
\@namedef{subjclassname@1991}{\textup{2020} Mathematics Subject Classification}
\makeatother

\begin{document}

\maketitle

\begin{abstract}
    Inspired by a question of Sarnak, we introduce the notion of a \emph{prime component} in an Apollonian circle packing:  a maximal tangency-connected subset having all prime curvatures.  We also consider \emph{thickened prime components}, which are augmented by all circles immediately tangent to the prime component.  In both cases, we ask about the curvatures which appear.  We consider the residue classes attained by the set of curvatures, the number of circles in such components, the number of distinct integers occurring as curvatures, and the number of prime components in a packing.  As part of our investigation, we computed and analysed example components up to around curvature $10^{13}$; software is available.
\end{abstract}

\setcounter{tocdepth}{1}
\tableofcontents

\section{Introduction}

An Apollonian circle packing, or ACP, is a fractal set in the plane, obtained by repeatedly adding circles into an initial constellation of three (Figure~\ref{fig:pack}).  Specifically, starting from a collection of four pairwise tangent circles $C_1,C_2,C_3,C_4$, one adds all circles that are simultaneously tangent to a subset of three of the original circles:  by a theorem of Apollonius, for any set of three pairwise tangent circles, there are exactly two  circles simultaneously tangent to that set\footnote{One views a line as a circle with infinite radius.}.  One now has eight circles and can continue the process using any new set of three mutually tangent circles in the picture.  Figure~\ref{fig:pack} depicts this process.

\begin{figure}[h]
  \begin{center}
  \includegraphics[width=5in]{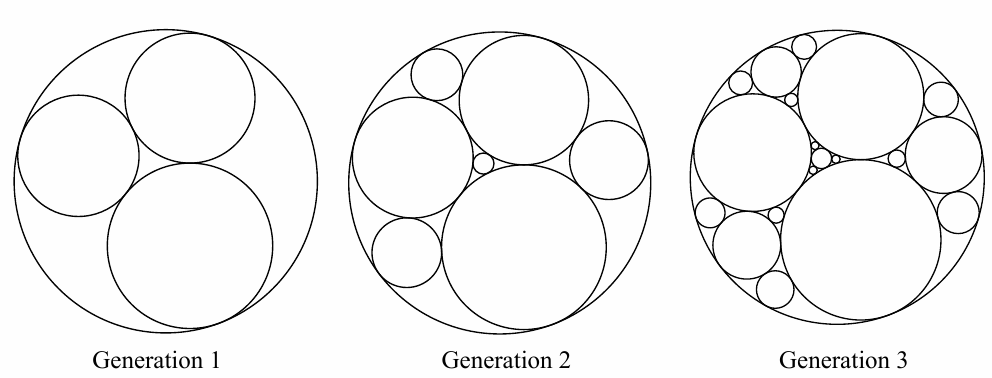}
  \caption{The construction of an Apollonian packing}
  \label{fig:pack}
  \end{center}
\end{figure}

Proceeding ad infinitum, one obtains a packing of infinitely many circles.  Remarkably, if any four pairwise tangent circles in the packing have integer curvature (the curvature is the reciprocal of the radius), all of the circles in the packing have integer curvature (Figure~\ref{fig:acp}). Throughout this paper, every integral packing we consider is further assumed to be \emph{primitive}, meaning that the greatest common divisor of all the curvatures of circles in the packing is $1$. This simple integrality fact, first observed by Nobel prize laureate in chemistry Frederick Soddy in the early 1900's, inspires many fascinating number theoretic questions about Apollonian packings: for example, which integers appear as curvatures in a given packing?  Are there infinitely many primes in any given such packing? 

\begin{figure}[h]\label{fig:intro packing}
  \begin{center}
  \includegraphics[width=5in]{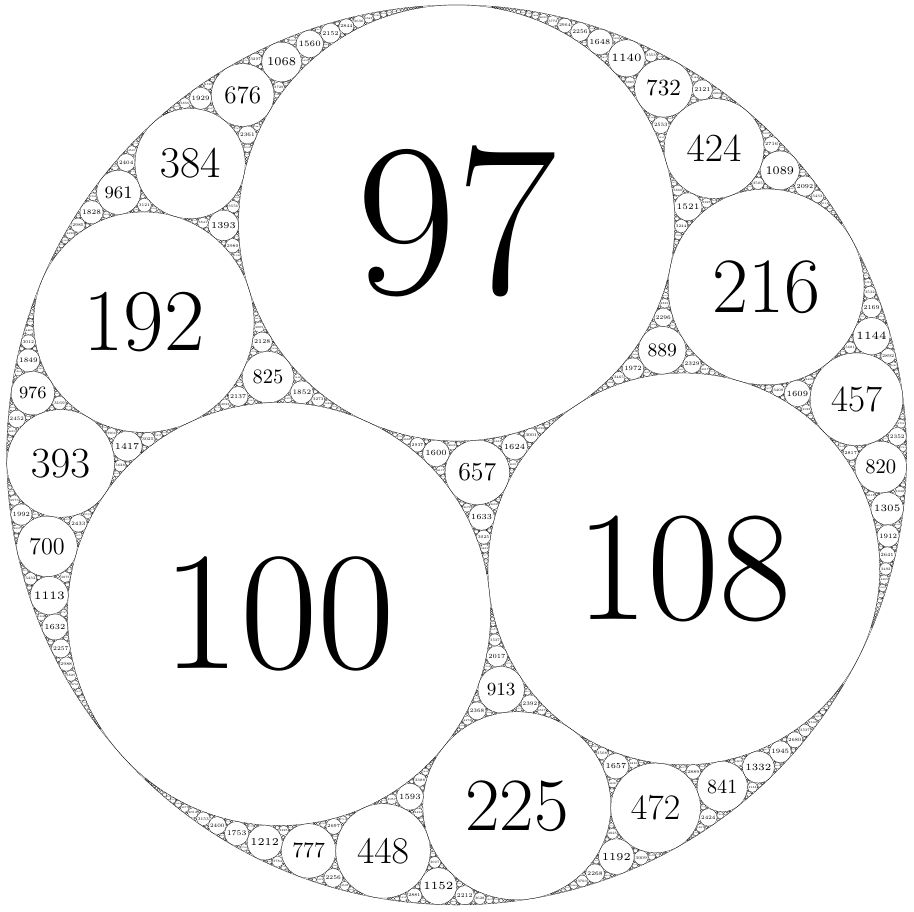}
  \caption{Circles of curvature $\leq 30000$ in the Apollonian circle packing corresponding to $(-47, 97, 100, 108)$.}
  \label{fig:acp}
  \end{center}
\end{figure}

It is known that the number of circles with curvatures $\le X$ is asymptotic to $c X^{1.3056\ldots}$, for some constant $c$ which depends on the choice of packing \cite{KontorovichOh}.  In an integral packing, then, the average multiplicity of individual curvatures approaches infinity.  It is therefore natural to ask:  do all integers appear?  In fact they do not, since it is known that any packing may entirely avoid certain residue classes modulo $24$ \cite{FuchsSanden}.  We call a curvature \emph{admissible} if it is allowed by congruence obstructions.  Then, according to the philosophy that this \emph{local} (congruence) obstruction is the only obstruction, it was conjectured by \cite{GLMWY} and \cite{FuchsSanden} that all sufficiently large admissible integers will appear.  This is called the \emph{local-to-global} conjecture. 
 It was only very recently discovered that there is also a type of obstruction arising from quadratic reciprocity, which rules out families of the form $cn^2$ or $cn^4$ amongst curvatures \cite{HKRS}, disproving the conjecture.  The current conjecture now has the following form.

\begin{conjecture}[{Apollonian curvature conjecture \cite[Conjecture 1.5]{HKRS}, modifying \cite{GLMWY},\cite{FuchsSanden}}]\label{ACPconj}
Let $\Pfull$ be a primitive integral packing, and let $\Sigma$ be the union of the residue classes modulo $24$ that have representatives in $\Pfull$.  Let $S$ be the union of the curvature families ruled out by reciprocity obstructions catalogued in \cite[Theorem 2.5]{HKRS}.  Then every sufficiently large integer $m$ with $m\in \Sigma \setminus S$ occurs as a curvature in $\Pfull$.
\end{conjecture}

This conjecture remains the central goal in the field.  In this paper, we consider certain subsets of the Apollonian circle packing, and ask the same questions.  
\begin{itemize}
    \item  A \emph{prime component} is  a maximal tangency-connected subset of circles whose curvatures are all odd primes. 
\end{itemize}  
Indeed, not a lot is known about circles of prime curvature in an Apollonian packing.
Sarnak first defined such components, after observing that, given any circle $C$ in the packing, there is a positive definite binary quadratic form $f_C(x,y)$ such that the curvatures of all circles tangent to $C$ in $\mathcal P$ are exactly the integers primitively represented by $f_C(x,y)-a_C$, where $a_C$ is the curvature of $C$ \cite{Sarnak}.  This implies that any fixed circle has infinitely many tangent circles of prime curvature, and consequently, prime components are infinite (see Proposition \ref{prop:prime-comp-infinite}).  Sarnak points out that one gets infinite trees of circles connected by tangencies, all of prime curvature. 

Evidently, prime components cannot satisfy Conjecture~\ref{ACPconj} as stated, but we might ask, do all sufficiently large primes appear?  
There is a long list of natural questions to ask about prime components in an ACP: 
\begin{itemize}
    \item \emph{Local questions:}  Modulo a given integer $d$, do the primes in a given prime component mimic the primes in the whole packing? 
    \item \emph{Growth with multiplicity:}  Given a real number $X$, how many circles of curvature at most $X$ are there in a fixed prime component?
    \item \emph{Positive density or local-to-global (growth without multiplicity):}  How many primes less than $X$ appear as curvatures in a fixed prime component?  Do a positive density of primes appear?  All but finitely many?  (Reciprocity obstructions do not apply to primes.)
    \item \emph{Number of components:}  Can we count prime components in a given ACP?
\end{itemize}

\begin{figure}[h]
  \begin{center}
  \includegraphics[width=5in]{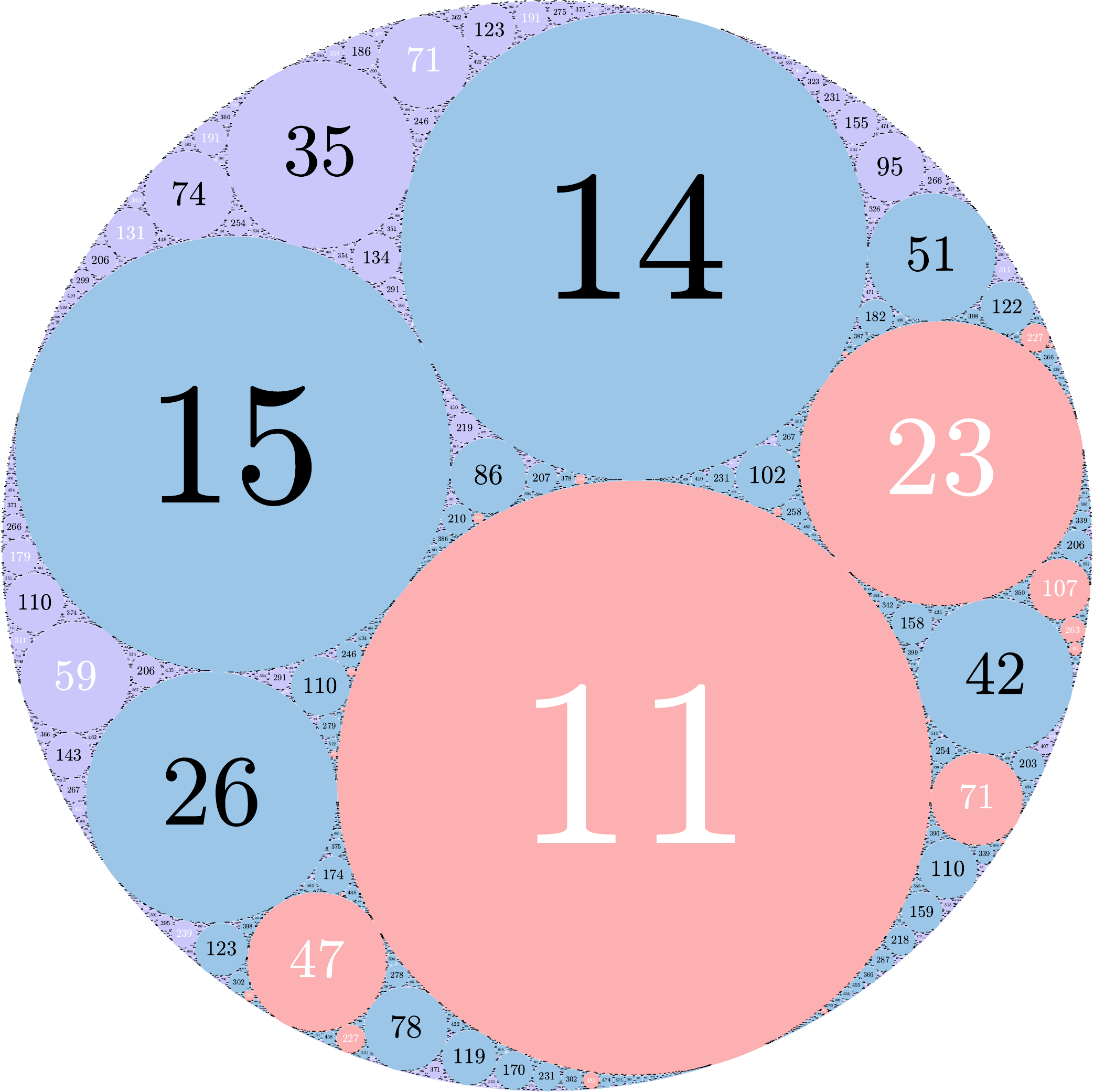}
  \caption{The packing informally known as the \emph{People's packing}, up to curvature 10,000.  In pink, a prime component.  In blue, the additional circles in the thickened prime component.  All other circles are purple.}
  \label{fig:prime-and-thick}
  \end{center}
\end{figure}

In an attempt to reach closer to existing tools, it is natural to consider also thickened prime components.  

\begin{itemize}
    \item A \emph{thickened prime component} is a prime component augmented by adding every circle which is tangent to it.
\end{itemize}  
These might, na\"ively, be expected to attain all the integers attained by the ambient packing.  Do they?
One can hope that, if the arithmetic structure of thickened prime components is rich enough, we might be able to answer some questions:
\begin{itemize}
    \item \emph{Local questions:}  Modulo a given integer $d$, do the curvatures in a thickened prime component mimic the curvatures in the whole packing?
    \item \emph{Growth with multiplicity:}  For a real number $X$, how many circles of curvature less than $X$ are there in a fixed thickened prime component?
    \item \emph{Positive density:}  Do a positive density of all integers appear as curvatures in a fixed thickened prime component? 
    \item \emph{Curvature conjecture}:  Do the curvatures in a given thickened prime component satisfy a conjecture similar to Conjecture~\ref{ACPconj} in the whole packing?
\end{itemize}

Although we address this wide variety of questions, let us state the analogue to Conjecture~\ref{ACPconj} as the motivating conjecture for the entire paper:

\begin{conjecture}
    \label{conj:the}
    \begin{enumerate}
        \item  Let $\Ppr$ be a prime component.  Then the set of curvatures appearing in $\Ppr$ is the same as the set of primes appearing in the ambient packing $\Pfull$, up to finitely many exceptions.
        \item Let $\Pth$ be a thickened prime component.  Then the set of curvatures appearing in $\Pth$ is the same as the set of curvatures appearing in the ambient packing $\Pfull$, up to finitely many exceptions.
    \end{enumerate}
\end{conjecture}

This conjecture remains open.  Among our varied results, we obtain, for example, the following:

\begin{theorem}\label{thm:primesdensitycorollary}
Let $\Pth$ be a thickened prime component in a primitive integral Apollonian packing.  Then the number $\curth$ of distinct integers less than $X$ occuring in $\Pth$ satisfies
$$\curth \gg \frac{X}{(\log \log X)^{1/2}}.$$
\end{theorem}
Throughout the paper, the notation $g \gg f$ means that $g$ is bounded below by some positive constant multiple of $f$.

\subsection{Local-to-global for Apollonian packings}
Let us begin with some facts about Apollonian packings as a whole.  Progress towards proving Conjecture~\ref{ACPconj} (in its various forms) dates back to \cite{GLMWY}, who first examined the local properties of Apollonian packings.  They took advantage of certain parameterized sets of circles to show that, if $d$ is relatively prime to $30$, every residue class modulo $d$ appears among the curvatures in a primitive ACP. Fuchs gave a stronger condition, showing that the modulus 24 captures all possible congruence obstructions \cite{Fuchs}.  

Shifting from local to local-to-global, let $\curfull$ denote the number of distinct integers less than $X$ occurring as curvatures in the packing $\Pfull$. Then Conjecture \ref{ACPconj} would imply 
\begin{equation}\label{ratio}
\curfull= \frac{r_{\Pfull}}{24}X +O(\sqrt{X}),
\end{equation} 
where $r_\Pfull \in \{6,8\}$ is the number of residue classes modulo 24 which appear in the packing.

A first lower bound on $\curfull$ that comes relatively close to the expected growth was obtained by Sarnak \cite{Sarnak}.

\begin{theorem}[\cite{Sarnak}]
We have 
\[
\curfull\gg\frac{X}{\sqrt{\log X}}.
\]
\end{theorem}

This theorem follows from Sarnak's observation on quadratic forms.  It is this observation that led to a series of results, each coming closer to the desired local to global conjecture. Building on this observation, Bourgain-Fuchs showed that a positive proportion of integers occurs in an integral Apollonian circle packing.

\begin{theorem}[\cite{BourgainFuchs}]\label{bourgain-fuchs} We have
\[\curfull\gg X.\]
\end{theorem}

Taking this method further, Bourgain and Kontorovich showed that the local-global conjecture holds for almost all positive integers $m\leq X$ \cite{BourgainKontorovich}, a result which has later been generalised to a larger class of Kleinian groups in \cite{FSZ}.  
\begin{theorem}[\cite{BourgainKontorovich}]
For some $\epsilon > 0$,
\[
\curfull= \frac{r_\Pfull}{24}X +O(X^{1-\epsilon}).
\]  
\end{theorem}
Though this falls just short of \eqref{ratio}, it has a nice consequence concerning \emph{prime} curvatures: since $X^{\epsilon}$ exceeds $\log(X)$, it implies positive density of primes as curvatures in a primitive Apollonian packing. This is not apparent from techniques in \cite{BourgainFuchs} alone, but has already been proved by Bourgain in \cite{BourgainACPprime}.

The study of prime components presents one novel new difficulty.  The curvatures of circles in the whole ACP can be viewed as coordinates of vectors in an orbit of a linear group called the \emph{Apollonian group}.  This group acts on quadruples of curvatures of four pairwise tangent circles in the packing.  By contrast, neither the prime components nor the thickened prime components constitute a full group orbit.     Our results rely heavily on the surviving tool: Sarnak's observation about the relationship of curvatures in the packing and integers represented by shifted binary quadratic forms.

Our main results are as follows. 

\subsection{Local results.}  In Sections \ref{sec:congruence classes} and \ref{sec:revisiting local}, we investigate local obstructions in prime and thickened prime components (meaning congruence restrictions with respect to a modulus $m$). 

In Section~\ref{sec:congruence classes}, the methods we use depend on Sarnak's original observation in \cite{Sarnak} that certain families of curvatures appearing in a primitive Apollonian circle packing are represented by shifted binary quadratic forms.
In particular, our first result will depend on some statements concerning primes represented primitively by shifted quadratic forms which appear in \cite{paper2}; we state these results in Theorem~\ref{thm:maintheorempaper2}.  Using this theorem, we prove the following:

\begin{theorem*}
  [Theorem~\ref{thm:primeclasses}]
   Let $m\in\ZZsub{>1}$ be coprime to $6$, and let $\ell \in (\ZZ/m\ZZ)^\times$. Then, a prime component $\Ppr$ (respectively a thickened prime component $\Pth$) contains infinitely many primes congruent to $\ell \pmod m$ among those circles within two tangencies of any sufficiently small fixed circle.  
\end{theorem*}

In Section \ref{sec:revisiting local}, we revisit the question of local obstructions in prime and thickened prime components from the perspective taken in \cite[Section 6]{GLMWY}.  In place of binary quadratic forms, we construct special types of matrix words in the Apollonian group, written in terms of the generators $S_1,S_2,S_3$, and $S_4$, and show, conditional on a well-known conjecture of Bunyakovsky, that the orbits of a Descartes quadruple under these special matrix words contain circles with curvatures in all residue class for moduli $m$ coprime to $30$.  One interest of this approach is that instead of bounding the number of tangencies to reach a given residue class, we study the length of the word in the Apollonian group (in other words, the path in the Cayley graph).  This is closer in spirit to the way strong approximation is proven for the Apollonian group.

\begin{theorem*}[Theorem~\ref{thm:2-step walk-prime}] 
The two-tangency path of Theorem~\ref{thm:primeclasses} above corresponds to a walk in the Cayley graph whose length is bounded in terms of Bunyakovsky's Conjecture~\ref{conj:bunyakovsky}.
\end{theorem*}

\subsection{Size of prime components}
As far as growth of circles \emph{with multiplicity} in a prime component, we can only show a trivial lower bound (Proposition~\ref{prop:growth-primecomp}).  Let us write
\begin{align*}
\cirpr &:= \# \{ C \in \Ppr : \operatorname{curv}(C) \countingupperbound X \}, \\
\cirth &:= \# \{ C \in \Pth : \operatorname{curv}(C) \countingupperbound X \}.
\end{align*}
We make the following conjecture, based on experimental data:
\begin{conjecture}
\label{conj:growth-primecomp}
Let $\Ppr$ be a prime component.  Then
    \[
    \lim_{X \rightarrow \infty}
    \frac{\cirpr}{\pi(X)} = \infty
    \]
    Let $\Pth$ be a thickened prime component.  Then
    \[
    \lim_{X \rightarrow \infty}
    \frac{\cirth}{X}  = \infty
    \]
\end{conjecture}

\subsection{Counting integers in prime components}

Combining Conjecture~\ref{ACPconj} and Conjecture~\ref{conj:the}, we expect all but finitely many admissible primes to occur in $\Ppr$; and for $\Pth$, up to finitely many exceptions, to be subject to only congruence and reciprocity obstructions, as for the whole packing.  More precisely, one expects firstly
\[
N_\Ppr(X) 
=
\sum_{a \text{ admissible}} \# \{ p \text{ prime} < X : p \equiv a \bmod{24} \} + O(1)
\sim \frac{s_\Pfull}{8}\pi(X),
\]
where $s_\Pfull \in \{1,2\}$ is the number of invertible admissible residue classes (see \cite[Proposition 1.2]{HKRS} for a classification) and $8 = \varphi(24)$, and secondly,
\[
N_\Pth(X) = \frac{r_\Pfull}{24}X + O(\sqrt{X})
\]

In Section \ref{sec:towards}, we show the first lower bound on integers appearing in a thickened prime component:  Theorem~\ref{thm:primesdensitycorollary} stated above.
This uses the methods of \cite{BourgainFuchs}, and it appears that this is the best one can do with these methods alone. The key idea in \cite{BourgainFuchs} is to consider all integers that are curvatures of circles two levels of tangencies away from a fixed circle.  These are regarded as integers represented by an infinite family of shifted binary forms $f_a(x,y)-a$. A crucial part of this is to balance how many of these forms one considers (one needs enough to get a positive fraction of integers being represented) with how many integers are represented by more than one form (one doesn't want this to overwhelm the positive fraction of integers obtained by summing over all the forms).  In our case every form must come from a circle of prime curvature, and there are simply not enough of these to obtain a positive fraction of integers while still keeping the size of the intersections small enough. 

Nonetheless, in forthcoming work, we replace primes with $r$-almost primes -- numbers with at most $r$ prime factors -- for sufficiently large $r$, and revisit positive density in thickened $r$-almost prime components.  

\subsection{Counting prime components}
Another question that arises naturally is the number of prime components in a single Apollonian packing.  It is not difficult to show there are infinitely many, but we provide a detailed heuristic count in Section \ref{sec:count-prime-comp}, leading to the following conjecture.
\begin{conjecture*}[Conjecture~\ref{PCRconjecture}]
The number of prime components in an Apollonian circle packing $\Pfull$ is asymptotic to 
    $$ c\frac{\cirfull}{\log X},
$$
where $c$ is an explicit constant.
\end{conjecture*}

\subsection{Experimental data.}  Finally, the paper includes an extensive computational investigation of the questions mentioned above, and accompanying software is available (see below).  In particular, curvatures with multiplicity in prime and thickened prime components were computed up to around $10^{13}$, and the results support Conjecture~\ref{conj:growth-primecomp}; we postulate some estimated growth rates derived na\"ively from exploring the data.  We also compute the frequency of curvature multiplicities, which have some unexpected and unexplained irregularities.  Considering the slow growth rate expected of prime components, it is not surprising that a large proportion of integers are still of multiplicity zero in this range.  Finally, we examine Conjecture~\ref{PCRconjecture} concerning the number of prime components, finding it plausible.  The data, although extensive, is probably best interpreted as not having `settled down' to its asymptotic behaviour.

\subsection{A note on the images and software.} 
The PDF images here, in their original form, are of sufficient quality that an arbitrary-zoom-capable PDF viewer will allow the user to read any curvature up to the stated bound.  Images were created using publicly available open source software created by the fifth author \cite{GHapollonian} which runs with Pari/GP \cite{PARI2}, with the aid of Sage Mathematics Software \cite{SageMath} in the case of colour images.

The experimental data was collected using C and PARI/GP code, and methods to replicate the results can be found in \cite{GHapollonianPrime}. 

\subsection{Paper structure.}
The paper is structured as follows. In Sections \ref{sec:background}-\ref{sec:count-prime-comp}, we go over the background on Apollonian packings and prime components and provide some heuristics towards counting prime components in an Apollonian circle packing. Sections \ref{sec:congruence classes} and \ref{sec:revisiting local} are devoted to local properties of prime and thickened prime components. Section \ref{sec:towards} contains the proof of Theorem~\ref{thm:primesdensitycorollary}. In Section \ref{sec:data}, we present data illustrating our results and supporting conjectures we make throughout the paper.

\subsection{Acknowledgements.}
This project originated as part of the Women in Numbers 4 workshop held at the Banff International Research Station. We are grateful to the organizers Jennifer Balakrishnan, Chantal David, Michelle Manes, and Bianca Viray for creating a welcoming and stimulating research environment throughout the workshop. We also thank the Mathematical Sciences Research Institute, which hosted us in 2019 as part of their SWiM program. We all thank Peter Sarnak for inspiring this research with his questions about prime components in Apollonian circle packings. Hsu thanks the University of Bristol and the Heilbronn Institute for Mathematical Research for their partial support of this project. Rickards and Stange thank James McVittie and Drew Sutherland for insights into the experimental data analysis.  Thank you to Alex Kontorovich for feedback.  And finally, we thank the anonymous referee for insightful comments.

\section{Background}
\label{sec:background}

\subsection{Descartes quadruples and the Apollonian group}\label{sec:descapgroup}

A natural way to think about the structure of the packing is to study its quadruples.  A \emph{Descartes quadruple} of circles is a collection of four circles which are pairwise tangent to one another. A \emph{Descartes quadruple} of numbers is the quadruple of curvatures of a Descartes quadruple of circles. In this paper, we use this term both in reference to quadruples of circles and quadruples of curvatures.

Indeed, having determined one Descartes quadruple of a packing, one can find all quadruples in the packing as an orbit of the first. This is a consequence of the following theorem, which is commonly attributed to Descartes, but appeared in correspondence with Princess Elisabeth of Bohemia \cite[\textit{The Correspondence}]{Shapiro2007-SHATCB-6}. 

\begin{thm}[Descartes, Princess Elisabeth of Bohemia, 1643] Let $a,b,c,d$ be curvatures of four pairwise tangent circles, where we take the curvature of a circle internally tangent to the other three to have negative curvature. Then 
\begin{equation}
    \label{eqn:desc}
    Q(a,b,c,d):=2(a^2+b^2+c^2+d^2)-(a+b+c+d)^2=0
\end{equation}
\end{thm}

Thus, one can move from one quadruple to another, by fixing three of the variables in the equation above:  given a Descartes quadruple $(a,b,c,d)$ in a packing, one has that $(a',b,c,d)$ is another Descartes quadruple in that packing, where 
$$a+a'=2(b+c+d).$$
From a more geometric point of view, according to Apollonius there are exactly two circles tangent to the circles of curvature $b, c$ and $d$, and they are exactly the ones of curvature $a$ and $a'$.  If the original quadruple lies in a particular Apollonian circle packing, this `swapping' process (swapping out one circle for its alternate) results in a new quadruple in the same packing.

There are four possible `swaps' of this type. 
If $\mathcal A=\langle S_1,S_2,S_3,S_4\rangle$ is the subgroup of $\textrm{O}_Q(\mathbb Z)$ generated by 
\begin{equation*}\small{
S_1=\left(
\begin{array}{llll}
-1&2&2&2\\
0&1&0&0\\
0&0&1&0\\
0&0&0&1\\
\end{array}
\right),}\quad
\small{
S_2=\left(
\begin{array}{llll}
1&0&0&0\\
2&-1&2&2\\
0&0&1&0\\
0&0&0&1\\
\end{array}
\right),}
\end{equation*}
\begin{equation*}
\small{
S_3=\left(
\begin{array}{cccc}
1&0&0&0\\
0&1&0&0\\
2&2&-1&2\\
0&0&0&1\\
\end{array}
\right),\quad
S_4=\left(
\begin{array}{cccc}
1&0&0&0\\
0&1&0&0\\
0&0&1&0\\
2&2&2&-1\\
\end{array}
\right)},
\end{equation*}
then the set of all Descartes quadruples in a packing containing the quadruple $\mathbf v=(a,b,c,d)$ is exactly the orbit $\mathcal A\cdot\mathbf v^T$. This was most likely first discovered by Hirst \cite{Hirst}, who called $\mathcal A$ the \emph{Apollonian group}. 

In \cite{GLMWY}, it is shown that every packing has a unique \emph{root quadruple} of curvatures, which essentially correspond to the four largest pairwise tangent circles in the packing. There are exactly two integral ACP's where this root quadruple corresponds to more than one Descartes quadruple of circles in the packing: the one generated by the quadruple $(-1,2,2,3)$, where there are two Descartes quadruples with these curvatures, and the one generated by $(0,0,1,1)$, where there are infinitely many quadruples with these curvatures. In these two cases, one simply makes a choice to call just one of these quadruples of circles the root, and all results in \cite{GLMWY} hold regardless of the choice.

More generally, we have the following result about how circles in a packing correspond to vectors in $\mathcal A\cdot\mathbf v^T$.

\begin{lemma}[\cite{GLMWY}, Theorem~3.3]
\label{lemma:onebig}
 Let $\Pfull$ be an Apollonian circle packing with root $\mathbf v$. Also, let $(\mathcal{A}\cdot\mathbf{v}^T)'$ denote the set $\mathcal{A}\cdot\mathbf v^T$ with  $\mathbf{v}^T$ removed. Then, the set of circles in $\Pfull$ that are not in the root quadruple are in one to one correspondence with $(\mathcal{A}\cdot\mathbf{v}^T)'$ through the following relation: the curvatures of the circles in $\Pfull$, counted with multiplicity, consist of the four elements of $\mathbf v$ plus the largest elements of each vector in $(\mathcal{A}\cdot\mathbf{v}^T)'$. 
\end{lemma}
We have the following useful corollary:

\begin{lemma}
\label{lemma:onebirth}
    If $C \in \Pfull$ is a circle not in the root quadruple, then $C$ appears as the largest curvature in exactly one Descartes quadruple of $\Pfull$.
\end{lemma}

We call this the \emph{birth quadruple} of $C$.  

Moreover, given a Descartes quadruple $\mathbf{v}$, if $S_{i_1}S_{i_2}\cdots S_{i_k}$ is a reduced word in the generators $S_i$ above, meaning $i_j\not=i_{j+1}$ for any $j$, it is shown in the proof of  \cite[Theorem 3.3]{GLMWY} that the $i_1$-th coordinate of $(S_{i_1}S_{i_2}\cdots S_{i_k})\cdot\mathbf{v}^T$ is the largest of all the coordinates, and in particular is larger than the $i_1$-th coordinate of $(S_{i_2}S_{i_3}\cdots S_{i_k})\cdot\mathbf{v}^T$.

\subsection{Local properties}

Let $S$ be the set of curvatures of a primitive Apollonian circle packing $\Pfull$.  Then, when $\gcd(m,6)=1$ the set of residues of $S$ modulo $m$ is known to be the full set of residues modulo $m$.  In fact, any existing congruence restriction on $S$ is captured by listing the residue classes attained modulo $24$, see \cite{Fuchs}.  The possible residue classes modulo $24$ are classified in \cite[Proposition 2.1]{HKRS}.  The following more precise statement modulo $8$ will sometimes be useful.

\begin{proposition}
\label{prop:descmod}
    Let $\Pfull$ be an Apollonian circle packing.  Every Descartes quadruple has two even curvatures and two odd curvatures.  Furthermore, exactly one of the following is true:
    \begin{enumerate}
        \item all odd curvatures in $\Pfull$ are $1 \pmod{8}$; or
        \item all odd curvatures in $\Pfull$ are $5 \pmod{8}$; or
        \item all odd curvatures in $\Pfull$ are $3,7 \pmod{8}$, and tangent odd curvatures are not equivalent modulo $8$.
    \end{enumerate}
\end{proposition}

\begin{proof} This follows from \cite[Lemma 3.3]{HKRS}.
\end{proof}

\subsection{Quadratic families}\label{quadfam}

Much of our paper utilizes the following observation of Sarnak's on quadratic forms related to Apollonian packings.  The following was first observed in \cite{Sarnak}, although we provide a proof here for completeness. 

\begin{proposition}[{\cite{Sarnak}}]
\label{prop:QF}
    Let $C_a$ be a circle in $\Pfull$, having curvature $a$ and lying within a Descartes quadruple $(a,b,c,d)$.  Let $S$ be the set of circles tangent to $C_a$.  Define
   \[
     f_a(x,y) - a = (b+a)x^2+(a+b+d-c)xy+(d+a)y^2 - a.
    \]
    Then, the multiset of curvatures of the circles in $S$ is exactly the multiset of values $f_a(x,y)-a$ for $(x,y)$ coprime integers.
\end{proposition}

An example is given in Figure~\ref{fig:quadfam}.

\begin{figure}
    \centering
    \includegraphics[width=3in]{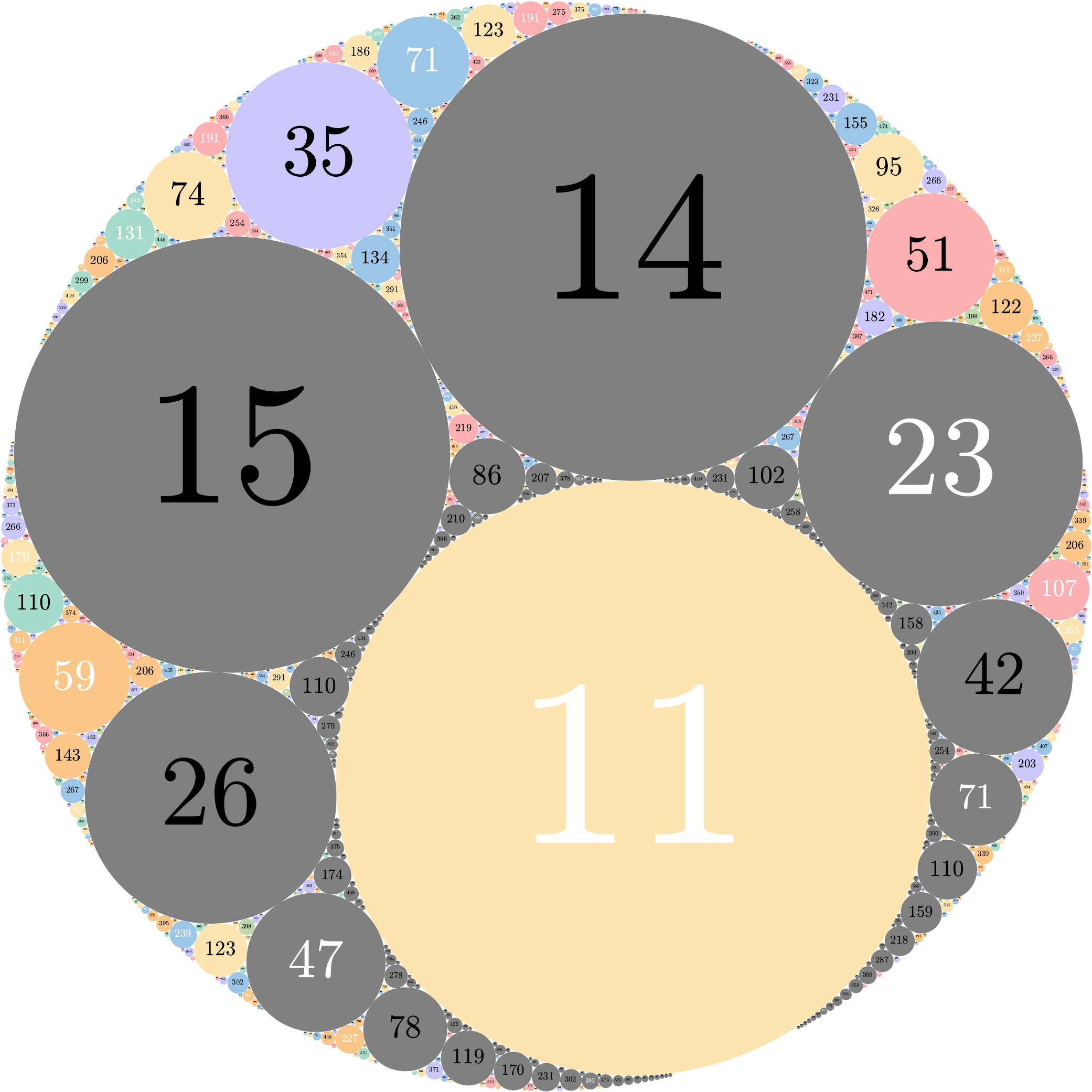}
    \caption{A quadratic family (all circles tangent to the circle of curvature 11) in the People's packing.  The quadratic family is shown in grey.  For circles not in the quadratic family, colours indicate residue class modulo $7$.  Font colour distinguishes primes and composites.}
    \label{fig:quadfam}
\end{figure}

\begin{proof}
   Let $\mathcal{A}_1$ denote the subgroup of $\mathcal{A}$ generated by $S_2,S_3$, and $S_4$, i.e., $\mathcal{A}_1 = \langle S_2, S_3, S_4 \rangle$. This result is obtained by noting that the circles tangent to $C_a$ correspond to exactly those circles obtained by acting on the quadruple $(a,b,c,d)$ by $\mathcal{A}_1$. More specifically, recall from \cite{Sarnak} and \cite{BourgainFuchs} that the orbit $\mathcal{A}_1\cdot(a,b,c,d)^T$ can be expressed as the set of points
\begin{equation}
\left(
\begin{array}{c}
a\\
A_0(2k+1)^2+2B_0(2k+1)(2m)+C_0(2m)^2-a\\
A_0(2k+1-2l)^2+2B_0(2k+1-2l)(2m-2n-1)+C_0(2m-2n-1)^2-a\\
A_0(2l)^2+2B_0(2l)(2n+1)+C_0(2n+1)^2-a\\
\end{array}
\right)
\end{equation}
where $A_0=b+a, \, 2B_0=a+b+d-c, \, C_0=d+a$, and $k,l,m,n$ are such that
\begin{equation}
\left(
\begin{array}{ll}
2k+1&2l\\
2m&2n+1\\
\end{array}
\right)\in\textrm{SL}_2(\mathbb Z).
\end{equation}
and thus, correspond to the integer values of $f_a(x,y)-a$ for coprime integers $(x,y)$.
\end{proof}

An important ingredient in our analytic considerations are bounds concerning primes represented by shifted quadratic forms. 
The following is a special case of a result of Iwaniec.

\begin{theorem}[{\cite[Theorem 1]{Iwaniec72shiftedprimes}}]\label{thm:shiftedprimes}
Let $f$ be a primitive positive definite integral binary quadratic form, with discriminant $\Delta$ not a perfect square.  Let $\omega$ be a non-zero integer.  Write $b_{f,\omega}(n)$ for the characteristic function for whether $n$ is represented  by $f(x,y)-\omega$.  Then
\[
\sum_{\substack{p \le X,\\ p \text{ prime}}} b_{f,\omega}(p)
\gg \ll
\frac{X}{(\log X)^{3/2}}.
\]
\end{theorem}
This result uses the notation $g \gg \ll f$, which denotes that $g$ is bounded both above and below by positive constant multiples of $f$.

Sarnak claimed (without proof) in \cite{Sarnak} that Iwaniec's result can be modified to give the same estimate for primitively represented forms.  In Section~\ref{sec:congruence classes}, where we investigate residue classes of primes appearing in a fixed prime component, we also need an analogous statement for primes in congruence classes.  (As an aside, Theorem~14.7 of \cite{FriedlanderIwaniec} is exactly of this flavor, but for integers represented as a sum of two squares, as opposed to shifted primes represented by a quadratic form.)

In a forthcoming paper, a subset of the present authors generalize Iwaniec's results to include primitivity and congruence classes, simultaneously.  Here we state only the result in the case of curvature forms like those in Proposition~\ref{prop:QF}.

\begin{theorem}[{\cite[Theorem 1.3]{paper2}}]
\label{thm:maintheorempaper2}
Let $f_a$ be a primitive positive definite integral binary quadratic form of the form given in Proposition~\ref{prop:QF}, in particular, having discriminant $\Delta = -4 a^2$. Assume also that $a$ is odd.  Write $b_{f_a,a}(n)$ for the characteristic function for whether $n$ is represented primitively by $f_a(x,y)-a$.  
\begin{enumerate}
\item  Then \[
\sum_{\substack{p \le X,\\ p \text{ prime}}} b_{f_a,a}(p)
\gg \ll
\frac{X}{(\log X)^{3/2}}.
\]

\item Let $\ell$ and $m$ be coprime integers satisfying $(m,2\Delta)=1$ and $(\ell + a, m) = 1$.  Then
\[
\sum_{\substack{p \le X,\\ p \text{ prime}\\p\equiv \ell\bmod{m}}} b_{f_a,a}(p)
\gg 
\frac{X}{(\log X)^{3/2}},
\]
where the implicit constant may depend on $m$. 
\end{enumerate}
\end{theorem}

The main idea of the proof in \cite{paper2} is to obtain a result on the number of square-free shifted primes $<X$ represented by the genus of $f_a(x,y)$, which turns out to be still bounded above and below by a positive constant multiple of $X/(\log X)^{3/2}$. Note that all of these are represented primitively. This entails combining a sieving argument with Iwaniec's argument in \cite{Iwaniec72shiftedprimes}. We then use the method of Bourgain--Fuchs in \cite{bourgain2012representation} to show that of these square-free shifted primes, those which are represented by some but not all forms in the genus (i.e. potentially not represented by $f_a$) are few enough that one can conclude the lower bound in the first statement of Theorem~\ref{thm:maintheorempaper2}. The second statement of Theorem~\ref{thm:maintheorempaper2} requires incorporating the additional condition on $p$ into Iwaniec's argument.

\subsection{Pinch families}
In our investigation of the local properties of prime components, it is useful to consider smaller subsets of the packing called \emph{pinch families}.

\begin{definition}\label{defn:pinch}
Let $C_a$ and $C_b$ be two tangent circles in $\Pfull$.  
    The \emph{pinch family} they govern is the family of circles $C \in \Pfull$ which are simultaneously tangent to $C_a$ and $C_b$.  
\end{definition}

An example is given in Figure~\ref{fig:pinch}.

\begin{figure}
    \centering
    \includegraphics[width=3in]{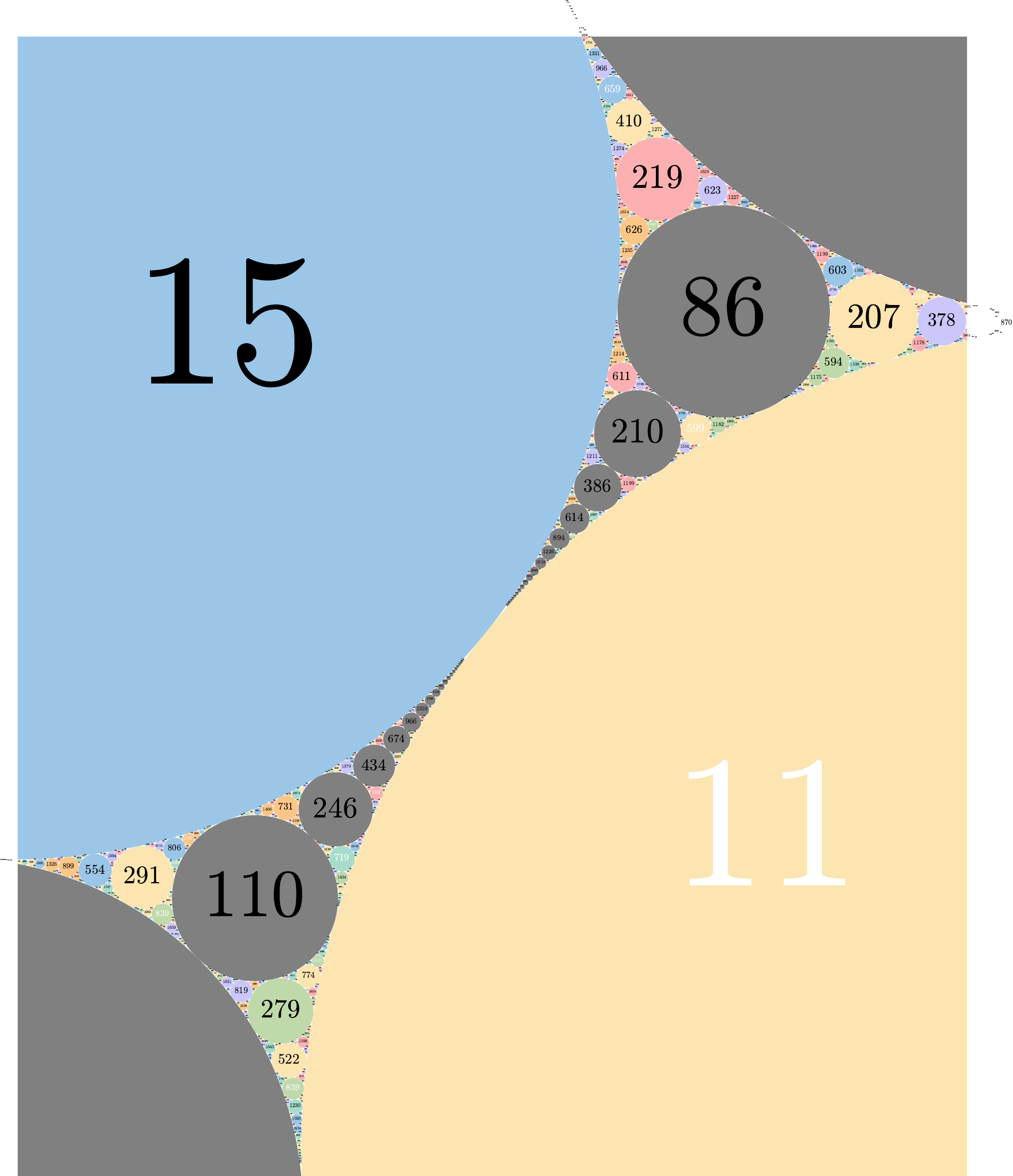}
    \caption{A pinch family in the People's packing.  The pinch family is shown in grey.  For circles not in the pinch family, colours indicate residue class modulo $7$.  Font colour distinguishes primes and composites.}
    \label{fig:pinch}
\end{figure}

The following lemma states that the curvatures of circles appearing in a pinch family are parameterized by a single-variable quadratic polynomial.  This is to be expected, since, as compared to the quadratic families parameterized by a binary quadratic form, we are now fixing two circles instead of one.

\begin{lemma}[{\cite[Lemma 4.5]{HKRS}}]\label{lem:curvatures in pinch fam}
Let $(a,b,c,d)$ be a Descartes quadruple, where curvatures $a,b$ correspond to circles $C_a$ and $C_b$, respectively. The curvatures of pinch family of circles tangent to both $C_a$ and $C_b$ is parametrized by 
\[f(x)=(a+b)x^2-(a+b+c-d)x+c,\,\quad \,x\in\mathbb{Z}. 
\]
In fact, $c = f(0)$, $d=f(1)$ and we have 
\[
(S_4S_3)^{s}\cdot (a,b,c,d)^T = (a,b,f(4s),f(4s + 1))^T,\;\;s\in\ZZ.
\]
\end{lemma}
For future reference, using Eq.~\eqref{eqn:desc}, we may compute
\begin{equation}
    \label{eqn:disc-f}
    \operatorname{Disc}(f) = 4ab.
\end{equation}

Our methods in Section \ref{sec:revisiting local} require the following lemma that describes the possible residue classes in a pinch family in terms of the residue classes of the governing circles $C_a$ and $C_b$.

\begin{lemma}\label{lem:Fm}
Let $m = p^n$ be an odd prime power, and continue the notation of Lemma~\ref{lem:curvatures in pinch fam}, in particular the function $f(x)$.  Consider the set
\[
F_m := \{ f(x) : x \in \ZZ/m\ZZ \}.
\]
Let $Q_m$ denote the set of squares modulo $m$.
Then there are three cases:
\begin{enumerate}
    \item If $(a,b) \equiv (0,0) \pmod{p}$, then $c \equiv d \pmod{m}$, and $F_m = \{ c \}$, and $\# F_m = 1$.
    \item Otherwise, if $p \mid a+b$, then $p \nmid a+b+c-d$, $F_m = \ZZ/m\ZZ$, and hence $\#F_m = m$.
    \item Otherwise, if $a+b$ is invertible,
    \[
F_m = (a+b)Q_m - \frac{ab}{(a+b)},
\] and $\# F_m = \# Q_m$.
\end{enumerate}
In particular, the set $F_m$ depends only on $a$ and $b$ (not $c$ or $d$).
\end{lemma}

\begin{proof}
If $a + b \equiv 0 \pmod{p}$, then $c \not\equiv d \pmod{p}$ unless we have a quadruple of the form $(0,0,c,c) \pmod{p}$.  Thus, we are either the first or second case whenever $p \mid a+b$.

For the second case, let $k \in \ZZ_p$.  Then $f(x)-k$ is linear modulo $p$, and therefore has a root modulo $p$.  Provided that $a+b+c-d \neq 0 \pmod p$ (which is the first case, so not this case), then $f'(x) \neq 0$ modulo $p$.  This implies that $f(x)-k$ has a root modulo $p^k$ also.  

    Otherwise, $a+b \not\equiv 0 \pmod{m}$.  By completing the square, since $2(a+b)$ is invertible modulo $m$,
    \begin{equation*}\label{eqn:pinchcomplete}
    f(X) = (a+b) X^2 + c - \frac{(a+b+c-d)^2}{4(a+b)}, 
    \end{equation*}
    where $X = x + \frac{a+b+c-d}{2(a+b)}$.  Applying Eq.~\eqref{eqn:desc}, we simplify to
    \[
    f(X) = (a+b) X^2 - \frac{ab}{(a+b)}.
    \]
\end{proof}

The elementary fact that a quadratic polynomial has infinitely many composite values has as a consequence the following.

\begin{lemma}\label{lemma:pinchcomposite}
    A pinch family contains infinitely many composite circles.
\end{lemma}

The question of whether a pinch family contains infinitely many primes (or any primes) is a special case of the following conjecture of Bunyakovsky (also spelled Bouniakowsky); see \cite{AZ} for a modern overview.

\begin{conjecture}[Bunyakovsky's conjecture]\label{conj:bunyakovsky} Suppose that  $f(x)\in\ZZ[x]$ is a quadratic polynomial that satisfies the following three conditions:
\begin{enumerate}
    \item the leading coefficient is positive,
    \item $f(x)$ is irreducible,
    \item the infinite sequence $f(1),\, f(2),\, f(3),\ldots$ has no common factor.
\end{enumerate}
Then $f(n)$ is prime for infinitely many positive integers $n$.
\end{conjecture}

Note that this conjecture is quite out of reach at this point, at least in the sense that there is no single known polynomial in one variable of degree $2$ which represents infinitely many primes.

\begin{remark}
    Our study of residue classes appearing in prime and thickened prime components can be split into two strategies: one is to look at all circles tangent to a fixed one, which is where Theorem~\ref{thm:maintheorempaper2}\bs{(2)} comes into play, and the other is to travel along chains of circles tangent to two fixed ones, which is where Conjecture~\ref{conj:bunyakovsky} is relevant. The latter gives more precise information, in a sense, about where one finds a given residue class, but the former relies on a conjecture which is likely much easier to prove.
\end{remark}

\section{Prime components}
\label{sec:prime-comp}

In this section, we give the basic definitions and geometric properties of prime components.  By convention, negative curvatures are not considered prime.

\subsection{Basic properties of prime components}
\begin{definition}
    Let $\mathcal{P}$ be a primitive integral Apollonian circle packing.  Let $C_p$ be a circle of prime curvature.  
    
    \begin{itemize}
        \item The \emph{prime component} containing $C_p$ is the largest tangency-connected subset of $\mathcal{P}$ containing $C_p$ and consisting entirely of circles of odd prime curvature.
        \item The \emph{thickened prime component} containing $C_p$ is the prime component containing $C_p$, together with every circle tangent to it.
    \end{itemize}
\end{definition}

The birth quadruple of the largest prime curvature of a prime component can be thought of as the root of that component.

\begin{definition}
\label{defn:PCR}
    A \emph{prime component root} is a Descartes quadruple $(a,b,c,d)$ for which $d$ is prime, $a$, $b$ and $c$ are composite, and $d \geq a,b,c$. 
\end{definition}

Such a quadruple is illustrated in Figure~\ref{pcrpic}.

\begin{proposition}
\label{prop:prime-comp-infinite}
    Any prime component contains infinitely many circles.
\end{proposition}

\begin{proof}
This follows from Theorem~\ref{thm:maintheorempaper2}\bs{(1)}, which implies that any circle of odd curvature in the component is tangent to infinitely many prime circles.
\end{proof}

\begin{figure}[h]
\centering
\includegraphics[height = 32 mm]{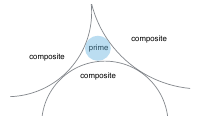}
\caption{A typical prime component root in its birth quadruple}\label{pcrpic}
\end{figure}

\begin{lemma}
\label{lemma:uniquePCR}
    Fix an Apollonian packing $\Pfull$.  With at most one exception, any prime component contains a unique prime component root.
\end{lemma}

\begin{proof}
    Suppose the component does not include a circle of curvature $2$.  There are at most two circles of curvature $2$ in a packing (this extremal case being achieved by the packing $(-1,2,2,3)$), and they are tangent, so there is at most one component of this type.  Suppose also that the component does not intersect the root quadruple of the packing.  At most one component can intersect the root quadruple, and in the case that the packing contains a circle of curvature $2$, this coincides with the component containing that circle.  Thus we have ruled out at most one prime component.

    For any other prime component, consider the smallest prime $p$ and a circle $C_p$ of that curvature in the component.  Consider all quadruples containing $C_p$.  Among these, since $C_p$ is not in the root quadruple, there exists one for which $a,b,c \leq d=p$ (Lemma~\ref{lemma:onebig}).  By assumption, $a,b,c$ are composite.  This is a prime component root.

    It remains to consider whether a component can have two such roots, involving distinct circles $C_{p,1}$ and $C_{p,2}$.  The circle $C_{p,i}$ and therefore the prime component is contained in the triangle bounded by the other three circles of its prime component root, which are composite.  Since the circles are distinct, these triangles are distinct.  Hence the circles cannot be part of the same prime component.
\end{proof}

Observe that a prime component root is essentially the same information as the smallest bounding composite triangle for the prime component.

\begin{lemma}
\label{lemma:prime-in-triangle}
    Suppose $C_a$, $C_b$, $C_c$ are mutually tangent circles in $\Pfull$.  Then within the triangle they create, there is a circle of prime curvature.  Hence, within any mutually tangent triangle of circles of composite curvature, there lies a prime component.
\end{lemma}

\begin{proof}
    Choose a circle of odd curvature within the triangle.  The family of curvatures tangent to it, which lie within the triangle, have curvatures equal to the values of a translated quadratic form, and as such, includes at least one primitively represented prime by Theorem~\ref{thm:maintheorempaper2}\bs{(1)}.
\end{proof}

Assuming that Conjecture \ref{conj:bunyakovsky}, we can strengthen Lemma~\ref{lemma:prime-in-triangle} as follows:

\begin{theorem}
\label{thm:prime-in-triangle}
Let $\Pfull$ be a primitive integer ACP, and consider any circle $C_b$ in the packing of curvature $b\in\mathbb Z$.   Assume Conjecture \ref{conj:bunyakovsky}. If $C_c$ and $C_d$ (of curvatures $c$ and $d$, respectively) are two touching circles that are both tangent to $C_b$ and satisfy $c\not\equiv d\pmod 2$, then there is a circle $C_p$ of prime curvature $p$ which is tangent to $C_b$ and lies within the triangular interstice bordered by $C_b, C_c,$ and $C_d$ as in the first picture in Figure~\ref{findprime}. 
\end{theorem}

\begin{proof}

Let $C_a$ be the unique circle (of curvature $a$) inside this interstice which is tangent to $C_b, C_c,$ and $C_d$, and let $\mathbf{v}$ denote the Descartes quadruple $\mathbf{v}=(a,b,c,d)$.  With $\mathcal{A}_1$ defined as in Proposition \ref{prop:QF},  the orbit $\mathcal{A}_1\cdot\mathbf{v}^T$
is precisely the collection of curvatures of all quadruples of pairwise tangent circles in $\Pfull$ which include $C_a$.  Without loss of generality, assume that $b>a,c,d$, so that $C_b$ is the only circle of curvature $b$ tangent to $C_a$. (If this is not the case, replace $C_c$ and $C_d$ by two smaller tangent circles which are also tangent to $C_b$ in the interstice between $C_c$ and $C_d$ until no circle of curvature $b$ can fit into the resulting interstice.)  Hence, to prove this theorem, it suffices to find Descartes quadruples $(a,b,x_3,x_4)\in\mathcal{A}_1\cdot\mathbf{v}^T$ with at least one of $x_3,x_4$ prime, since these quadruples will yield a circle of prime curvature in the region depicted in the second picture in Figure~\ref{findprime}, tangent to both $C_b$ and $C_a$. But, finding quadruples $(a,b,x_3,x_4)$ with at least one of $x_3,x_4$ prime is equivalent to finding circles of prime curvature in the pinch family governed by $C_a$ and $C_b$. Assuming Conjecture \ref{conj:bunyakovsky}, one can show that the quadratic form that parameterizes the curvatures appearing in this pinch family (defined in Proposition \ref{lem:curvatures in pinch fam}) represents infinitely many prime values, and if this is the case, there are infinitely many circles of prime curvature in the region depicted in the second picture in Figure~\ref{findprime} that are tangent to the circle $C_b$. 
\end{proof}
\begin{figure}[htp]
\centering
\includegraphics[height = 40 mm]{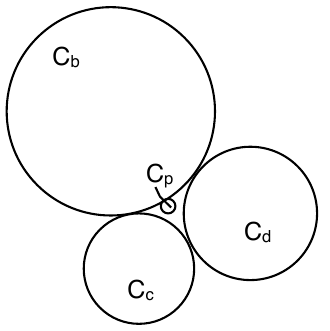} \quad \includegraphics[height = 40 mm]{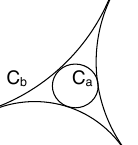}
\caption{Locating prime circle between tangent circles}\label{findprime}
\end{figure}

\begin{remark}
    This theorem can be thought of as a statement about topological density of circles of prime curvature and hence of prime components in a given packing.
\end{remark} 

As we see in Section \ref{sec:revisiting local}, our ability to locate circles of prime curvature within a fixed pinch family plays an important role when analyzing the local structure of prime and thickened prime components.

\subsection{The size of prime and thickened prime components}
\label{sec:size-prime-comp}

We now briefly consider the number of circles in a prime or thickened prime component. Recall that
    \[
    C_\mathcal{P}(X) := \# \{ C \in \mathcal{P} : \operatorname{curv}(C) \countingupperbound X \},
    \]
    where $\mathcal{P}$ is either a prime or thickened prime component.  This counts the circles in $\mathcal{P}$ including multiplicity.  In order to motivate the study in this paper, we conjecture the following (Conjecture~\ref{conj:growth-primecomp}):
            For a prime component,
\[
\lim_{X \rightarrow \infty} \frac{C_\Ppr(X)}{\pi(X)} = \infty,
\]
and for a thickened prime component,
\[
\lim_{X \rightarrow \infty} \frac{C_\Pth(X)}{X} = \infty.
\]
These statements would imply that the average multiplicity of a curvature in either packing is increasing.

Theorem~\ref{thm:maintheorempaper2}\bs{(1)}, implies an immediate lower bound on the growth of a prime component:

\begin{proposition}
\label{prop:growth-primecomp}
    Every prime component $\Ppr$ contains infinitely many circles.  In fact,
    \[
    \cirpr
    \gg \frac{X}{(\log X)^{3/2}},
    \]
    where the implied constant depends on the prime component.
\end{proposition}

\begin{proof}
    Let $C_p$ be any odd prime circle of the prime component.  Then, by Theorem~\ref{thm:maintheorempaper2}\bs{(1)}, there are $\gg X/(\log X)^{3/2}$ prime curvatures in the family of circles tangent to $C_p$.
\end{proof}

Conjecture~\ref{conj:growth-primecomp} is supported by experimental data; see Sections~\ref{sec:growthrateprime} and \ref{sec:growthratethick} for estimates on the exact growth rates.

\section{Counting prime components}
\label{sec:count-prime-comp}

In this section, we show that there are infinitely many prime components, prove an upper bound, and give a heuristic estimate and conjecture for the number of such components.

\begin{theorem}
There are infinitely many prime components.
\end{theorem}

\begin{proof}
Choose two circles $C_a$ and $C_b$ in the packing $\Pfull$ which are composite and tangent.  For the existence of such, take a quadruple with sufficiently large curvatures and choose the two even curvatures (see Proposition~\ref{prop:descmod}).  Consider the pinch family governed by these two circles (Definition~\ref{defn:pinch}).  It contains infinitely many circles of composite curvature $D_1, D_2, \ldots$ by Lemma~\ref{lemma:pinchcomposite}.  Within each triangle $C_a$, $C_b$, $D_n$ of composite curvatures, there exists a prime component, by Lemma~\ref{lemma:prime-in-triangle}.  For any finite list of such components, we can choose $n$ large enough that the triangle $C_a$, $C_b$, $D_n$ does not contain it.  Therefore there are infinitely many such components.
\end{proof}

Let $\Pfull$ be a primitive integral Apollonian packing and let $\cirfull$ be the number of circles of curvature $\countingupperbound X$ in the packing.  It is known that $\cirfull \sim X^{\delta}$ where $\delta \approx 1.3056\ldots$ is the Hausdorff dimension of an Apollonian circle packing \cite{Boyd,KontorovichOh,LeeOh, Vinogradov2012EffectiveBE}.  Let
$$
N^{\text{root}}_\Pfull(X):=\#\{ (a,b,c,p) \mbox{ a prime component root}: a,b,c \leq p \countingupperbound X, p \text{ prime}\}
$$
count the number of prime component roots with prime not exceeding $X$.  By Lemma~\ref{lemma:uniquePCR}, this also counts the number of prime components.

We propose the following conjecture.

\begin{conjecture}\label{PCRconjecture}
$$
N^{\text{root}}_\Pfull(X) \sim c\frac{ \cirfull}{\log X} - c''\frac{\cirfull}{(\log X)^2}\sim c\frac{\cirfull}{\log X},
$$
where $c=\prod_{\substack{\text{$p$ prime}\\ p \equiv 1 \,(4)}} \frac{p^2}{p^2 - 1}  \prod_{\substack{\text{$p$ prime}\\ p \equiv 3 \,(4)}}\frac{p^2}{p^2+1}=0.9159\dots$, and $c''$ is a constant between $0$ and $2$.
\end{conjecture}

The best upper bound we can prove is the following.

\begin{theorem}\label{thm:PCRupperbd}
\begin{equation*}
N^{\text{root}}_\Pfull(X)\ll \frac{\cirfull}{\log X}
\end{equation*}
\end{theorem}

\begin{proof}
Let $\mathcal{Q}$ be the set of Descartes quadruples of $\Pfull$.  Consider the two quantities
\begin{align*}
\Sigma_1(X)&:= \# \{
\mathbf{v}\in\mathcal Q : \max(\mathbf v)\countingupperbound X \mbox{{ and prime}} 
\}\\
\Sigma_2(X)& := 
\# \{
\mathbf{v}\in \Sigma_1(X) : \mbox{two coordinates of $\mathbf v$ are prime}
\}.
\end{align*}

As an immediate consequence of Proposition~\ref{prop:descmod}, we have
\begin{equation}\label{PCRincexc}
N^{\text{root}}_\Pfull(X)= \Sigma_1(X) - \Sigma_2(X),
\end{equation}
Therefore, bounding $N^{\text{root}}_\Pfull(X)$ from above by $\Sigma_1(X)$, we use the upper bound on $\Sigma_1(X)$ from \cite[Theorem 1.4]{KontorovichOh}.
\end{proof}

We devote the rest of this section to giving a heuristic justification for Conjecture~\ref{PCRconjecture}.
We continue to use the notation of the proof of Theorem~\ref{thm:PCRupperbd}. 

One has  an upper bound
\[
\Sigma_2(X) \ll \frac{\cirfull}{(\log X)^2}
\]
from \cite[Theorem 1.4]{KontorovichOh}.   But the best known lower bound on the first sum is \[
\Sigma_1(X) \gg \frac{ X}{\log X}
\]
from the work of \cite{BourgainKontorovich}, which is quite far from what is needed.

However, heuristically we can do better. In \cite{FuchsSanden}, the authors consider
\begin{equation}
\psi_{\Pfull}(X) = \mathop{\sum_{a(C) \countingupperbound X}}_{a(C) \;\text{prime}} \log \bigl(a(C)\bigr)
\label{intropsi}
\end{equation}
where $C$ is a circle in the packing $\Pfull$ and $a(C)$ is its curvature. They give the following heuristic in their Conjecture~1.2:
$$\psi_{\Pfull}(X) \sim L(2, \chi_4) \cdot \cirfull$$
where $L(2, \chi_4)= 0.9159\dots$ is the value of the Dirichlet $L$-series at $2$ with character $\chi_4(p)= 1$ for $p\equiv 1\;(4)$ and $\chi_4(p)=-1$ for $p\equiv 3\;(4)$.

Classically, it is known that $\pi(X)\sim\frac{\psi(X)}{\log X}$, where $\psi(X) = \sum_{p \text{ prime}} \log p$, and we expect the same for $\pi_{\Pfull}(X)$, the number of circles of prime curvature less than $X$ in a packing $\Pfull$. In other words, we expect that

$$\Sigma_1(X) \sim
c \frac{\cirfull}{\log X}
$$ where $c=0.9159\dots$. 
\newcommand{\curv}{\operatorname{curv}}

Regarding $\Sigma_2$, a heuristic in Conjecture~1.3 of \cite{FuchsSanden} gives that 
\begin{equation}
    \label{eqn:FSheuristicpair}
\mathop{\sum_{(C,\,C')\in S}}_{a(C),\, a(C')\countingupperbound X}\log (a(C))\cdot \log (a(C'))\sim  c'\cdot \cirfull,
\end{equation}
where $S$ is the set of pairs of tangent circles in $\Pfull$ both of which have prime curvature, and $$c'=2 \cdot \prod_{p \equiv 1 \,(4)} \frac{p^4}{(p^2 - 1)^2} \prod_{p\equiv 3\,(4)}\frac{p^4}{(p^2+1)^2}\cdot \left(1-\frac{2}{p(p-1)^2}\right)=1.3808\dots.$$
This is the count of pairs of tangent prime circles that we need, except weighted by products of $\log$'s of the curvatures. In spirit, this suggests that without the $\log$ weights one obtains
$$
c'' \frac{\cirfull}{(\log X)^2}
$$ 
for some constant $c''$. Given this heuristic, the expression in (\ref{PCRincexc}) becomes
$$
N^{\text{root}}_\Pfull(X) \sim c\frac{ \cirfull}{\log X} - c''\frac{\cirfull}{(\log X)^2}\sim c\frac{\cirfull}{\log X},
$$ as desired. 

See Section~\ref{sec:data} for supporting data, which supports the fact that $c''$ is between $0$ and $2$.

\section{Local properties of a prime component}\label{sec:congruence classes}

In some sense the `first' question to address about a prime or thickened prime component is the existence of local obstructions.  Given a modulus $m$, which residue classes modulo $m$ are represented in the component?  In the case of a full packing $\Pfull$, the congruence restrictions can be understood by studying the Cayley graph of the group modulo $m$.  There is no natural way to insert primality into this type of analysis, so our methods rely on studying walks within the component, relying on results concerning the representation of primes by quadratic forms and quadratic functions.

\subsection{Primes represented by binary quadratic forms}

The following is a version of a standard fact about quadratic forms. See \cite[Chapter 4]{SerreCourse}, for example, for a more thorough discussion.

\begin{lemma}
  \label{lemma:primeclasses}
  Let $f(x,y)$ be an integral binary quadratic form with discriminant $\Delta\neq 0$, and fix $m\in\ZZsub{>1}$ coprime to $2\Delta$. Also, let $\ell\in \ZZ/m\ZZ$. Then, the following holds:
  \begin{itemize}
      \item if $\Delta\in(\ZZ/m\ZZ)^\times$ is a square, $f(x,y)\equiv \ell \pmod m$ has a non-zero solution;
       \item otherwise, $f(x,y)\equiv \ell \pmod m$ has a solution only for $\ell\in(\ZZ/m\ZZ)^\times$.
  \end{itemize} Moreover, if $f(x,y)\equiv\ell\pmod m$ has a non-zero solution, it also has a primitive solution.
\end{lemma}

\begin{proof} Working modulo $p$ for any prime dividing $m$, the first statement (without a primitivity condition) follows from \cite[Proposition 5, Ch. IV]{SerreCourse}. Moreover, since any such $p$ is coprime to $2\Delta$, a solution to $f(x,y)\equiv \ell\pmod p$, for $\ell\in\mathbb{F}_p$,  can be lifted via Hensel's lemma to a solution in $\ZZ_p$. Thus, applying Sun-tzu’s theorem, we obtain the first statement for any choice of $m\in\ZZ_{>1}$ coprime to $2\Delta$. 

It remains to show that given a non-zero solution, we can find a primitive solution. Suppose that $(x_0,y_0)\in\ZZ^2$ is a solution to $f(x,y)\equiv \ell\pmod m$. We first note that $\gcd(m,\gcd(x_0,y_0))=1$ since $(\gcd(x_0,y_0))^2\,|\,f(x_0,y_0)$.  Then, for any $k\in\ZZ$ satisfying 
\[
k\not\equiv -x_0m^{-1}\pmod p\;\;\text{ for }\;p\,\left|\,\frac{y_0}{\gcd(y_0,m)}\right.,
\]one can show that $(x_0+mk,y_0)$ is a primitive solution to $f(x,y)\equiv \ell\pmod m$, as desired. 
\end{proof}

For a fixed modulus $m$, we apply this lemma within the context of an Apollonian circle packing to give a description of the set of curvatures modulo $m$ tangent to a fixed circle: 
 
\begin{lemma}\label{lem:primeclassescircles}
Let $C_a$ be a circle in $\Pfull$, having curvature $a$ and lying within a Descartes quadruple $(a,b,c,d)$, and fix $m\in\ZZ_{>1}$ coprime to $2a$. Let $S_m(a)$ denote the set of curvatures tangent to $C_a$ modulo $m$.  This set depends only on $a$ modulo $m$ and the following holds:
\begin{itemize}
    \item if $-1\in(\ZZ/m\ZZ)^\times$ is a square, then $S_m(a)=\ZZ/m\ZZ$;
    
    \item otherwise, $S_m(a) = \{ b : a+b \in (\ZZ/m\ZZ)^\times \}$.
\end{itemize}
\end{lemma}

\begin{proof}
By Proposition \ref{prop:QF}, the circle $C_a$ has an associated translated quadratic form
   \[
     f_a(x,y) - a = (b+a)x^2+(a+b+d-c)xy+(d+a)y^2 - a
    \]
that parameterizes the curvatures tangent to $C_a$. In particular, the discriminant of $f_a(x,y)$ is $\Delta=-4a^2$, so the result follows from a direct application of Lemma \ref{lemma:primeclasses} with $f_a(x,y)$.
\end{proof}

We now use the relationship discussed in Section \ref{quadfam} between curvatures of circles in a prime component and primes primitively represented by shifted binary quadratic forms to study the residue classes of curvatures appearing in prime components and thickened prime components. The following is an immediate consequence of Theorem~\ref{thm:maintheorempaper2}\bs{(2)}.

\begin{proposition}
  \label{prop:prime-residue}
  Continuing with the notation from Lemma \ref{lem:primeclassescircles}, and supposing $a$ is odd, for any 
  \[\ell\in S_m(a)\cap(\ZZ/m\ZZ)^\times\cap\left( (\ZZ/m\ZZ)^\times-a\right),\] the collection of curvatures of circles tangent to $C_a$ contains infinitely many primes satisfying $p \equiv \ell \pmod m$.  In fact, the number of such primes below $X$ is $\gg \frac{X}{(\log X)^{3/2}}$.
\end{proposition}

\subsection{Congruence classes in prime components}
We now combine Lemma \ref{lem:primeclassescircles} and Proposition \ref{prop:prime-residue} to show that, using Theorem~\ref{thm:maintheorempaper2}\bs{(2)}, a prime component in an Apollonian circle packing contains infinitely many primes congruent to $\ell\pmod m$ for any invertible residue class $\ell\pmod m$.

\begin{theorem}
  \label{thm:primeclasses}
  Let $m\in\ZZsub{>1}$ be coprime to $6$ and $\ell \in (\ZZ/m\ZZ)^\times$. Then, a prime component $\Ppr$ (respectively, a thickened prime component $\Pth$) contains infinitely many primes congruent to $\ell \pmod m$ among those circles within two tangencies of any sufficiently small fixed circle.  In fact, the number of such primes below $X$ is $\gg \frac{X}{\log^{3/2} X}$. 
\end{theorem}

\begin{proof}
  Choose a circle $C_a$ of curvature $a>m$ in $\Ppr$ (throughout the proof we can also take a thickened component $\Pth$ instead).

  First, suppose $\ell + a$ is invertible modulo $m$. In this case, Proposition \ref{prop:prime-residue} immediately implies that there are infinitely many primes congruent to $\ell$ modulo $m$ among those circles tangent to $C_a$.

  It remains to consider the case where $\ell+a$ is not invertible modulo $m$. Here, it suffices to find any circle $C_p$ of sufficiently large prime curvature $p$ in $\Ppr$ such that $\ell+p$ is invertible modulo $m$, and then run the same argument again, this time considering circles tangent to $C_p$.  Indeed, we just showed that, tangent to $C_a$, there are circles of prime curvature $p$ congruent to any invertible residue class $k$ such that $k+a$ is invertible modulo $m$.  Hence, the result follows provided that, whenever $\ell+a$ is not invertible modulo $m$, there exists an invertible residue class $k$ such that $k+a$ and $\ell+k$ are both invertible modulo $m$. 

Assume $\ell+a$ is not invertible modulo $m$.  If $m$ is prime, then $\ell\equiv -a\pmod m$.  Let $k$ be a residue such that $k \not\equiv\pm a$ or $0\pmod m$, which exists provided $m \ge 5$.  Then $k+a$ and $\ell+k=k-a$ are invertible modulo $m$.  Note that if one can find a suitable $k$ in the case that $m=q$ is prime, then one can find a suitable $k$ modulo $q^n$ as well.  Hence by Sun-tzu’s theorem, for any $m$ coprime to $6$, one can find a suitable $k$, or, equivalently, a circle tangent to $C_a$ which is tangent to a circle of prime curvature congruent to $\ell \pmod{m}$.
\end{proof}

\begin{remark}Theorem~\ref{thm:primeclasses} does not provide a bound on the length of the word in the Cayley group of the Apollonian group that is required to reach the desired residue class.  In the next section we show how the length of this word depends on an effective form of Bunyakovsky's conjecture.
\end{remark}

\section{Revisiting the local structure of prime and thickened prime components}\label{sec:revisiting local}

In the previous section, we showed that, using Theorem~\ref{thm:maintheorempaper2}\bs{(2)}, all invertible residue classes modulo $m$, for $m\in\ZZ_{>1}$ coprime to $6$, appear in a prime or thickened prime component, within two tangencies of a starting circle.  In this section, we revisit this result from the perspective of explicit paths of bounded length through the Cayley graph.  This is perhaps in the spirit of a combinatorial spectral gap as studied for the Apollonian group (e.g., \cite{BourgainKontorovich,FSZ}).  

To be more precise, in this section, we work directly with the generating set of swaps $S_i$ for the Apollonian group $\mathcal{A}$ and restrict to matrix words built of units of the form $(S_iS_j)^k$. These types of matrix words lead to single variable quadratic polynomials and allow us to exploit certain algebraic properties of matrices in the Apollonian group. As a result of this shift in perspective, we require Conjecture \ref{conj:bunyakovsky} instead of Theorem~\ref{thm:maintheorempaper2}\bs{(2)}.

The main result of this section can be stated as follows: suppose $m$ is a positive integer coprime to 30, $\mathbf{v}=(a,b,c,p_0)$ is a Descartes quadruple with $p_0$ prime and $\gcd(c+p_0,m)=1$.  
We show that there exist integers $s_0$ and $r_0$ such that for
\[
\mathcal{W}_{s_0,r_0}(t):= 
\underbrace{\dots S_4S_3S_4S_3}_\text{$s_0$ letters}
\underbrace{\dots S_3S_2S_3S_2}_\text{$r_0$ letters}
\underbrace{\dots S_2S_1S_2S_1}_\text{$t$ letters},
\]
the union of the curvatures appearing in $\{ \mathcal{W}_{s_0,r_0}(t)\cdot \mathbf{v}^T \}_{t \in \ZZ}$ contains all residue classes modulo $m$.  By definition, circles contained in these quadruples are a subset of those within two tangency levels of $C_{p_0}$, but the subset is rather more restricted.  Furthermore, assuming additionally that $c+p_0>0$, $cp_0$ is not a perfect square, $a$ is odd, and Conjecture \ref{conj:bunyakovsky} holds, then for infinitely many $t$, the quadruple $\mathcal{W}_{s_0,r_0}(t) \cdot \mathbf{v}^T$ is contained in the thickened prime component containing $C_{p_0}$.

Our method for proving this is heavily inspired by the proof of \cite[Theorem 6.2]{GLMWY}. The most salient difference is that we work within two tangency levels of a fixed circle (rather than three), and the requirement to stay within prime or thickened prime components requires some new ingredients.

\subsection{Pinch families in Apollonian circle packings} 
Our main tool throughout this section is the study of pinch families as in Definition \ref{defn:pinch}.  The following notation will be helpful in discussing these collections of circles in more detail.

By a \emph{word}, we mean a matrix word written in the generators $S_1, S_2, S_3, S_4$ of $\mathcal{A}$. Following notation in \cite[\S 6]{GLMWY}, we define length-$s$ words of alternating form:
\[
W_{ji}(s) := \underbrace{\dots S_jS_iS_jS_i}_\text{$s$ letters}.
\]
We call these \emph{$s$-term swap products}. For odd $s$, $W_{ji}(s)$ is the identity matrix with entries in the $i$th and $j$th rows replaced by
\[
(i,k)\text{-th entry}=\begin{cases}-s &\text{if }k=i\\s+1 &\text{if }k=j\\s(s+1) &\text{else}
\end{cases}\;\;\;\;\text{ and } \;\;\;\;
(j,k)\text{-th entry}=\begin{cases}s &\text{if }k=j\\-(s-1) &\text{if }k=i\\s(s-1) &\text{else}
\end{cases}.
\]
For even $s$, we interchange the $i$th and $j$th rows in the above formul\ae.

Fixing a Descartes quadruple $\mathbf{v} = (c_1,c_2,c_3,c_4)$, fixing $i, j \in \{1,2,3,4\}$, and applying all words $W_{ji}(s)$, $s \in \ZZsub{\ge 0}$, the resulting new circles lie within a pinch family, i.e. the family of circles simultaneously tangent to $c_k$ and $c_\ell$, where $\{k,\ell\} = \{1,2,3,4\} \setminus \{i,j\}$.  When a fixed Descartes quadruple is understood, the term \emph{$W_{ji}$-pinch family} denotes this pinch family. 

Let $\mathbf{v}=(a,b,c,d)$ be a Descartes quadruple, and recall from Lemma \ref{lem:curvatures in pinch fam} that the the curvatures in the $W_{43}$-pinch family of $\mathbf{v}$, i.e., curvatures of circles tangent to both $C_a$ and $C_b$, are parameterized by 
\[f(x)=(a+b)x^2-(a+b+c-d)x+c,\,\,x\in\mathbb{Z},  
\]
with
\[
W_{43}(2s)\cdot(a,b,c,d)^T = (a,b,f(4s),f(4s + 1))^T,\;\;s\in\ZZ.
\]
As $s$ increases, this gives a formula for the mod $m$ residue classes in the $W_{43}$-pinch family. Note that while 
\[
f(x+m\ell)\equiv f(x)\pmod m,\;\;\;\forall\ell\in\ZZ,
\] 
the actual curvatures appearing in the $W_{43}$-pinch family always increase as $\ell$ increases. Moreover, conditional on Conjecture \ref{conj:bunyakovsky}, pinch families will contain infinitely many circles with prime curvature in each residue class appearing in the pinch family:

\begin{lemma}\label{lem:primes in pinch fam}
Continuing with the above notation, let $m$ be coprime to $6$ and $\ell$ modulo $m$ be an invertible congruence class represented by $f(x)$. Assume Conjecture \ref{conj:bunyakovsky}. If  $a+b>0$, $ab$ is not a perfect square, and $c$ is odd, then the $W_{43}$-pinch family contains infinitely many circles of prime curvature congruent to $\ell \pmod m$. 
\end{lemma}

\begin{proof}
We first note that under these assumptions on $a, b$ and $c$, using Eq.~\eqref{eqn:disc-f}, one can check that $f(x)$ satisfies the conditions in Conjecture \ref{conj:bunyakovsky}. Next, we fix some $x_0\in\ZZ_{\geq 1}$ such that $f(x_0)\equiv \ell\pmod m$ and 
consider the quadratic polynomial $g(k):=f(x_0+mk)$. Since the leading coefficient of $g(k)$ is $(a+b)m^2>0$ and $\{g(k)\,|\,k\in\ZZ_{\geq 1}\}\subseteq \{f(x)\,|\, x\in\ZZ_{\geq 1}\}$, $g(k)$ satisfies conditions (i) and (iii) in Conjecture \ref{conj:bunyakovsky}. Moreover, since 
\[\operatorname{Disc}(g(k)) = m^2 \operatorname{Disc}(f(x)) = 4m^2ab,\] $g(k)$ is irreducible and thus also satisfies the conditions in Conjecture \ref{conj:bunyakovsky}. We can therefore conclude that $g(k)=f(x_0+mk)$ is prime for infinitely many values of $k$, and each such $k$ corresponds to a circle in the $W_{43}$-pinch family of prime curvature congruent to $\ell\pmod m$.
\end{proof}

\begin{remark}
For $f(x)$ as in Lemma \ref{lem:primes in pinch fam}, the main result of \cite{IwaniecAlmostPrimes} implies there are infinitely many 2-almost primes in the $W_{43}$-pinch family of curvatures of $\mathbf{v}$. So, in the discussion that follows below, we could replace the role of prime components with 2-almost prime components and remove any dependence on Conjecture \ref{conj:bunyakovsky}.
\end{remark}

Now, by Lemma \ref{lem:Fm}, when $a+b\not\equiv 0\pmod{m}$, the pinch family governed by $C_a$ and $C_b$ cannot include all residue classes mod $m$. So, to use $s$-term swap products to construct an explicit collection of tangent circles that includes all residue classes mod $m$, we need to allow movement between different pinch families. As we see below, Lemma \ref{lem:primes in pinch fam} plays a pivotal role in ensuring we stay within a prime or thickened prime component during switches between pinch families. 

\subsection{Walks in prime and thickened prime components}

For the remainder of this section, our goal is to develop a method for constructing matrix words in the Apollonian group that give paths (of tangent circles) between specific residue classes modulo $m$. Specifically, we want an explicit matrix word in $\mathcal{A}$ that, for any invertible residue class modulo $m$, gives a path to any other residue class modulo $m$ satisfying the following properties:
\begin{itemize}
\item The initial circle in the path has prime curvature.
\item The path is contained within two tangency levels from the initial circle.
\item The path is contained within the thickened prime component of the initial circle.
\end{itemize}

We might try (na\"ively) to construct this type of matrix word using just three of the four generators of $\mathcal{A}$, i.e., construct paths contained within the first tangency level of the initial circle, but this is not possible by Lemma \ref{lem:primeclassescircles}. As such, we need to use all four generators $S_i\in\mathcal{A}$ and will necessarily swap our initial prime circle at some point in this process. The following generalization of a pinch family will play an important role in our method:

\begin{definition}\label{defn:W}
Let $\mathcal{W}=W_{j_1i_1}(s_1)\cdots W_{j_ki_k}(s_k)\in\mathcal{A}$ be a matrix constructed by concatenating $s$-term swap products, and let $\mathbf{v}_0=(a,b,c,d)$ be a Descartes quadruple. Let 
\[
\mathbf{v}_i^T=S_\ast\mathbf{v}_{i-1}^T,
\] where $S_\ast\in\{S_1,S_2,S_3,S_4\}$ runs over the matrices in $\mathcal{W}$ from right to left. Then, we define the $\mathcal{W}$\textit{-family of }$\mathbf{v}_0$ to be the collection of tangent circles that correspond to the curvatures appearing in the Descartes quadruples $\{\mathbf{v}_i: 0\leq i\leq s_1+\cdots+s_k\}$. 
\end{definition}

Although we use the word \emph{family}, this is a finite set (in contrast to a pinch family, which is infinite). Our approach is somewhat complicated by the fact that we wish to consider both paths in the Cayley graph of $\mathcal{A}$ (i.e. words or sequences of quadruples) and tangency paths of circles in the packing, which are related but not synonymous concepts.  Given a word $\mathcal{W}$ as above, we therefore develop some terminology for the various related sets of circles and paths of tangencies. 

\begin{definition}  With $\mathcal{W}$ as above, let $C_a$ and $C_{\ell}$ be circles in an ACP $\Pfull$.
\begin{itemize}
\item We say the circle $C_{\ell}$ is \emph{$\mathcal{W}$-tied} to $C_{a}$ if there exists a Descartes quadruple $\mathbf{v}_a$ containing $C_{a}$ such that $C_{\ell}$ is in the $\mathcal{W}$-family of $\mathbf{v}_a$. 
\item When $C_{\ell}$ is $\mathcal{W}$-tied to $C_{a}$, a $\mathcal{W}$\textit{-geodesic} from $C_a$ to $C_{\ell}$ is any minimal-length sequence of tangent circles appearing the $\mathcal{W}$-family of $\mathbf{v}_a$ that connects $C_a$ and $C_{\ell}$. Note that a $\mathcal{W}$-geodesic from $C_a$ to $C_{\ell}$ is not unique, but its length is.  An upper bound for the length of a $\mathcal{W}$-geodesic is $k$, i.e., the number of changes between pinch families in $\mathcal{W}$.
\item Given a prime component $\Ppr$ containing $C_a$, we say that a $\mathcal{W}$-geodesic from $C_a$ is a \textit{core geodesic} if it is contained in $\Ppr$, with the exception of the geodesic's terminal circle. This mean a core $\mathcal{W}$-geodesic from $C_a$ is contained in either the prime component $\Ppr$ or its thickened prime component.
\end{itemize}
\end{definition}

To summarize our goal in this new language, we want to construct a collection of words 
\[\mathcal{W}(s_1,s_2,s_3)=W_{j_1i_1}(s_1)W_{j_2i_2}(s_2)W_{j_3i_3}(s_3)\] such that as $s_1,s_2,$ or $s_3$ vary, the core $\mathcal{W}(s_1,s_2,s_3)$-geodesics tie circles of all residue classes modulo $m$ to $C_a$.  Since we have taken $k=3$ in Definition \ref{defn:W}, these core geodesics are at most length two.  We accomplish this goal by carefully choosing values for all $j_k,i_k,$ as well as $s_1$ and $s_2$ in such a way that the curvatures appearing in the third coordinate of $\mathcal{W}(s_1,s_2,s_3)\cdot\mathbf{v}^T$ can be represented as a linear polynomial in $s_3$. To ensure we stay within the thickened prime component under consideration, we use Lemma \ref{lem:primes in pinch fam} to find prime curvature circles when we switch between pinch families.  This restriction is the most difficult part of this process (and we only achieve it conditionally). 

We start by disregarding the primality aspect.  That is, by giving a method for constructing a collection of words $\mathcal{W}$ such that the set of (not necessarily core) $\mathcal{W}$-geodesics terminate in all residue classes, i.e. disregarding whether we stay in the thickened prime component.

\begin{proposition}\label{prop:2-step walk} Let $m\in\ZZsub{> 1}$ be coprime to $30$, and let $\mathbf{v}=(a,b,c,d)$ be a Descartes quadruple of curvatures in an ACP such that $\gcd(c+d,m)=1$. There exist integers $s_0$ and $r_0$ such that for
\[
\mathcal{W}_{s_0,r_0}(t):=W_{43}(s_0)W_{32}(r_0)W_{21}(t),
\]
the set of $\mathcal{W}_{s_0,r_0}(t)$-families of $\mathbf{v}$, as $t$ varies over $\ZZ$, contains all residue classes modulo $m$. In particular, for any residue class $\ell\pmod m$, we can find an integer $t_0$ such that there is a $\mathcal{W}_{s_0,r_0}(t_0)$-geodesic, of length at most two, that ties a circle of curvature $\ell\pmod m$ to $C_d$.  Furthermore, the word $\mathcal{W}_{s_0,r_0}(t_0)$ is length at most $5m$. 
\end{proposition} 

\begin{proof} This proof is largely inspired by the proof of \cite[Theorem 6.2]{GLMWY}. For $r,\,s\in\ZZsub{\geq 0}$, we consider $W_{43}(s)W_{32}(r) \in \mathcal{A}_1$. Setting $u:=rs+1$ and assuming that both $r$ and $s$ are odd, the third row of $W_{43}(s)W_{32}(r)$ is 
\[(u^2+s^2-1+su,-su,-rs+(r+1)s(s+1),1+s-rs(r-1)+r(r+1)s(s+1)).\]
The sum of the first and second coordinates is $u^2+s^2-1$ and the difference is $2su$.  To simplify the presentation, we wish to find $u$ and $s$ such that $u^2+s^2\equiv 1 \pmod m$ and $2su \not\equiv 0 \pmod m$. 

By assumption, we suppose $m > 5$. 
Let $p$ be prime with $p^w \mid\mid m$.  Necessarily, $p > 5$, and choosing $u_p = 3 \cdot 5^{-1}$, $s_p = 4 \cdot 5^{-1}$ modulo $p^w$ implies $u_p^2 + s_p^2 \equiv 1 \pmod{p^w}$ and $2s_pu_p \not\equiv 0 \pmod{p^w}$. These choices for $u_p,\,s_p$ give the relation $r_p\equiv -2\cdot 4^{-1}\pmod{p^w}$.
Using Sun-tzu’s theorem, we next find positive odd integers $S$ and $R$ such that $S \equiv s_p \pmod{p^w}$ and $R\equiv r_p\pmod{p^w}$ for all primes $p$ dividing $m$. Setting $U=RS+1$, we have
$U^2+S^2-1 \equiv 0 \pmod m$ and $2SU\not\equiv 0 \pmod m$, as desired. 

Now, let $A=SU$. Note that $\gcd(A,m)=1$, and we can rewrite the third row of $W_{43}(S)W_{32}(R)$ in the form $(A,-A,C,D) \pmod m$. Applying $S_2$ and $S_1$ alternating in succession to this row yields 
\[
((-1)^{t}A, (-1)^{t+1}A, -2tA+C,-2tA+D) \pmod m
\]
for $t \in \ZZsub{\ge 0}$. In other words, setting $s_0=S$ and $r_0=R$, the third row of the word
\[
\mathcal{W}(t)=W_{43}(s_0)W_{32}(r_0)W_{21}(t)\] is of the form
\begin{equation}\label{eq:special row}
(A,-A,-2tA+C,-2tA+D) \pmod m
\end{equation}for $t\in 2\ZZ$.
Thus, the $\mathcal{W}(t)$-family of $\mathbf{v}$ contains circles of curvatures in the residue class 
\begin{equation}\label{eq:res class}
-2A(c+d)t+A(a-b)+Cc+Dd\pmod m.
\end{equation}
Since $\gcd(2A(c+d),m)=1$, our conclusion follows as $t$ varies over $2\ZZ$. Note that if $t_0$ is chosen so that the residue class in Eq.~(\ref{eq:res class}) is congruent to some $\ell\pmod m$, then a length-two $\mathcal{W}(t_0)$-geodesic from $C_d$ to a circle $C_\ell$ of curvature $\ell\pmod m$ is given by 
\[
C_d\to\,C_{a'}\to\,C_\ell,
\] where $C_{a'}$ is the circle corresponding to the first coordinate in $W_{21}(t_0)\cdot\mathbf{v}$. 

The final statement of the proposition is a consequence of the fact that $s_0$, $r_0$ are only specified modulo $2m$, and $t_0$ modulo $m$.
\end{proof}

We now return to our original goal of studying the residual structure of prime and thickened prime components. To this end, we want to apply Proposition \ref{prop:2-step walk} starting from a Descartes quadruple containing a prime curvature. The first step is to check that an appropriate starting quadruple exists, which is guaranteed by the following lemma:
\begin{lemma}\label{lem:goodquad}
Let $m\in\ZZsub{> 1}$ be odd, and let $\mathbf{v}=(a,b,c,p_0)$ be a Descartes quadruple of curvatures
such that $p_0$ is prime and $a$ is odd. There is circle with curvature $c'>0$ in the $W_{32}$-pinch family of $\mathbf{v}$ such that $\gcd(c'+p_0,mp_0)=1$.
\end{lemma}

\begin{proof}
The proof is similar to the proof of the claim in \cite[Theorem 6.2]{GLMWY}, but we need a small modification since we want the fourth coordinate of $\mathbf{v}$ to stay fixed. For even $k\in\ZZ$, define $q(k)$ to be the sum of the third and fourth coordinates of $W_{32}(k)\cdot\mathbf{v}^T$, that is,
\[q(k):=k^2(a+d)+k(a-b+c+d)+d,\,\,\text{ for }k\in 2\ZZ.\]
Then, for any fixed odd prime $p$, at least one of the values of $q(2)$, $q(4)$, or $q(6)$ must be non-zero modulo $p$ or else the primitivity of $\mathbf{v}$ would be violated. The rest of the proof follows as in \cite[Theorem 6.2]{GLMWY}.
\end{proof}

\begin{theorem}\label{thm:2-step walk-prime} Assume the conjecture of Bunyakovsky, Conjecture \ref{conj:bunyakovsky}.  Let $m\in\ZZsub{>1}$ be coprime to $30$, and let $\ell$ represent a residue class modulo $m$.  Let $C_p$ be a circle of prime curvature, and let $\Pth$ be the thickened prime component containing $C_p$.  

Then there exists a circle $C_{\ell}$ at distance two tangencies from $C$ in $\Pth$ such that \[\operatorname{curv}(C_{\ell}) \equiv \ell \pmod{m}.\]  Furthermore, the two-tangency path is of the restricted form described in Proposition~\ref{prop:2-step walk}.
\end{theorem} 

\begin{proof}
    By Lemma~\ref{lem:goodquad}, we can chose a starting quadruple $\mathbf{v}=(a,b,c,p)$ containing $C_p$ such that $c+p>0$, $cp$ is not a perfect square, $a$ is odd, and Proposition~\ref{prop:2-step walk} applies.

We need only re-run the proof of Proposition \ref{prop:2-step walk} within the thickened prime component, verifying whether the $\mathcal{W}(t)$-geodesics that we construct are core geodesics. 
Because of how 
we have constructed the $\mathcal{W}(t)$-families of $\mathbf{v}$ in Proposition \ref{prop:2-step walk}, it suffices to examine the initial switch between the $W_{21}$ and $W_{32}$-pinch families. However, Lemma \ref{lem:primes in pinch fam} implies that, assuming Conjecture \ref{conj:bunyakovsky}, for any $t_0\in\ZZsub{>0}$, we can find some $k\in\ZZsub{\geq 0}$ such that $W_{21}(t_0+mk)$ has prime curvature. In particular, since
\[
W_{21}(t_0+mk)\equiv W_{21}(t_0)\pmod m,
\]
this means that we can construct a core $\mathcal{W}(t)$-geodesic, of length at most two, to any residue class modulo $m$ within the thickened prime component. We repeat this process and apply Lemma $\ref{lem:primes in pinch fam}$ at the end of $\mathcal{W}(t)$ if we also want the terminal circle of our $\mathcal{W}(t)$-geodesic to have prime curvature.
\end{proof}

\begin{figure}[h]
  \begin{center}
  \includegraphics[width=5in]{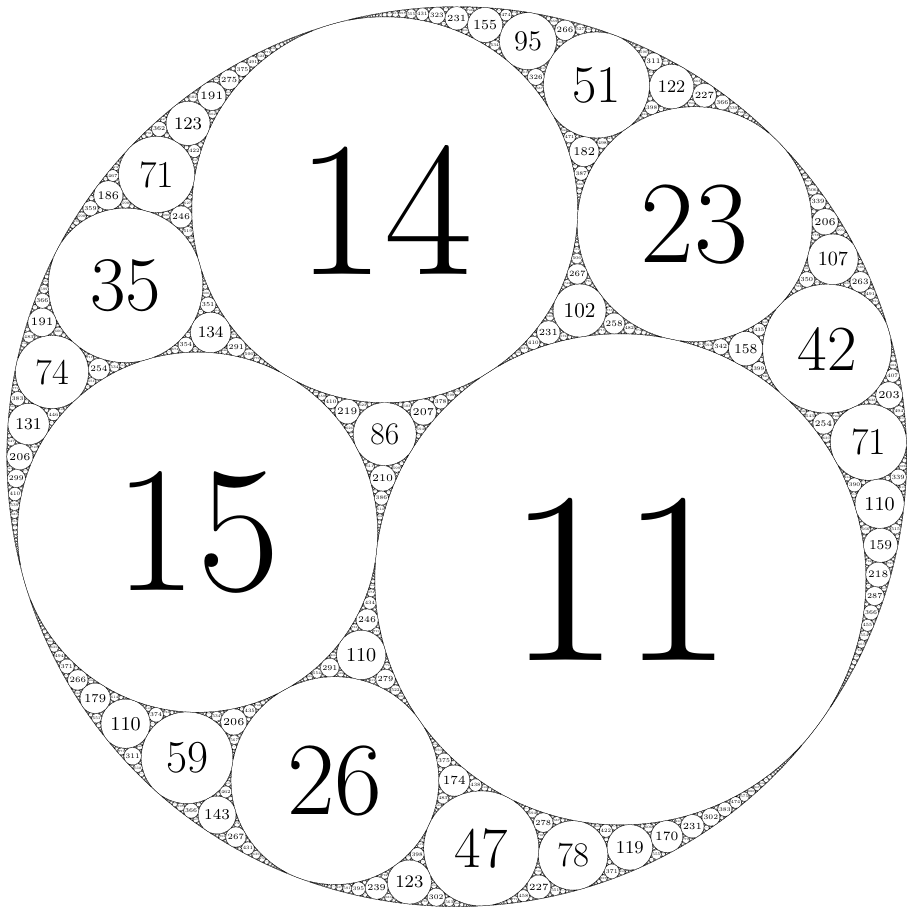}
  \caption{Circles of curvature $\leq 10000$ in the People's packing.}
  \label{fig:peoplepacking}
  \end{center}
\end{figure}

We conclude this section with a modification of Theorem \ref{thm:2-step walk-prime} that replaces the dependence on Conjecture \ref{conj:bunyakovsky} with Theorem~\ref{thm:maintheorempaper2}\bs{(2)} under some additional assumptions on the matrix word $W_{43}(s_0)W_{32}(r_0)$ defined in Proposition \ref{prop:2-step walk}. The key difference in our modification is that instead of searching for a circle of prime curvature within the $W_{21}$-family of the starting quadruple in Theorem \ref{thm:2-step walk-prime}, we instead arrange for one of the governing circles in this pinch family to have prime curvature. Under the additional hypotheses on $W_{43}(s_0)W_{32}(r_0)$, this means we only need to find a specific residue class within the $W_{21}$-pinch family to prove Theorem \ref{thm:2-step walk-prime}. 

The following lemma gives a Descartes quadruple that serves as the starting point for our modification of Theorem \ref{thm:2-step walk-prime}:
\begin{lemma}\label{lem:special quad!} Let $m\in\ZZ_{>1}$ be coprime to $30$, and let $k\in\ZZ$. For a circle $C_d \in \Pfull$ of prime curvature $d$ coprime to $m$, there exists a Descartes quadruple $\mathbf{v}=(a,b,c,d)$ containing $C_d$ such that $C_a$ has prime curvature and $a-b\equiv k\pmod m$.
\end{lemma}

\begin{proof}
To emphasize the difference between integers and residue classes, we use $\overline{x}$ to denote a residue class of $x$ modulo $m$. 
By Lemma~\ref{lem:Fm}, the set of residue classes modulo $m$ which appear in a pinch family with parent circles $C_a$ and $C_d$ is determined only by the residues $\overline{a}$ and $\overline{d}$.  We denote this set by $F_m(\overline{a},\overline{d})$.

The goal is to find circles $C_a$ and $C_b$ tangent to $C_d$ satisfying $a-b \equiv k \pmod{m}$ and $a$ prime.
By Proposition \ref{prop:prime-residue}, it suffices to find a pair of residue classes $(\overline{a},\overline{b})$ (which we will then lift to integer curvatures $a$ and $b$) such that 
\begin{enumerate}
    \item[(ia)] $\overline{a}$ is invertible modulo $m$,
    \item[(ib)] $\overline{a} + \overline d$ is invertible modulo $m$,
    \item[(ii)] $\overline{a}$ is represented by $f_{C_d}(x,y) -d$ modulo $m$ 
    \item[(iii)] $\overline{b} \in F_m(\overline{a},\overline{d})$, and
    \item[(iv)] $\overline{a}-\overline{b}\equiv k\pmod m$.
\end{enumerate}
It suffices to achieve this modulo $p^e$ for each maximal prime power dividing $m$, and then combine these with the Theorem of Sun-tzu. Thus, we assume $m$ is an odd prime power $p^e$.

By Lemma \ref{lem:Fm}, a residue class $\overline{b}$ appears in pinch families between circles with curvatures congruent to $\overline{a}$ and $\overline{d}$  if  the expression 
\begin{equation}\label{eq:quad residue}
\overline{a}\overline{d}+\overline{b}(\overline{a}+\overline{d})
\end{equation}
is in $Q_m$ (the set of squares; notation from~Lemma~\ref{lem:Fm}). Substituting  $\overline{b}= \overline{a}-\overline{k}$ into Eq.~(\ref{eq:quad residue}), we have that conditions (iii) and (iv) hold if the value at $\overline{a}$ of the polynomial
\[
g(z):=z^2+(2\overline{d}-\overline{k})z-\overline{k}\overline{d}.
\]
is in $Q_m$.  Thus, we just want to find such an $\overline{a}$ which satisfies (ia), (ib) and (ii).  

Observe that the value set $G_m$ of $g(z)$ includes 
\begin{equation}\label{eqn:gvals}
\begin{split}
g(-d) = g(k-d) &= -d^2, \\
g(k/2) = g(k/2 - 2d) &= -(k/2)^2,  \\
g(0) = g(k-2d) &= -dk,  \\
g(k) = g(-2d) &= dk. 
\end{split}
\end{equation}
\textbf{Case I:  $-1$ is a quadratic residue modulo $p$ for $p$ odd}.  Provided the set 
\[T := \{ -\overline d, \overline k- \overline d, \overline k/2, \overline k/2 - 2\overline d\}\] has at least three elements, we can choose $\overline a$ from this set so that $\overline a$ and $\overline a+ \overline d$ are coprime to $p$.  Thus, we may assume this set has at most two elements.  In this case, recalling that $(d,p)=1$, we have $\overline k \in \{ -2\overline d, 2\overline d\}$.  Therefore, either $T=\{ \pm \overline d \}$ or $T= \{ -\overline d, -3\overline d \}$.  Since $d$ is coprime to $p$, taking $\overline a = \overline d$ in the first case and $\overline a = -3\overline d$ in the second case will guarantee that $\overline a$ and $\overline a+\overline d$ are coprime to $p$.  
\\\\\noindent \textbf{Case II: $-1$ is a quadratic non-residue modulo $p$}.  
From Eq.~\eqref{eqn:gvals}, there exists $\overline{x}\in\ZZ/m\ZZ$ with $g(\overline{x}) \in Q_m$:  either $\overline x \in \{0, \overline k-2\overline d\}$ or $\overline x \in \{\overline k, -2\overline d\}$ will do.  By assumption, $\overline{d}$ is non-zero modulo $p$.  Therefore $g(-\overline d)=-\overline d^2$ is not a square modulo $p$, which implies $\overline{x}+\overline{d}$ is not divisible by $p$, hence invertible modulo $p$. Thus, by Lemma~\ref{lem:primeclassescircles}, we have satisfied (ii) and (ib). It remains to show that we can find such an $\overline{x}$ that is invertible modulo $p$.  If $\overline x = -2\overline d$ we are done, so the remaining case is $\overline x \in \{0,\overline k-2\overline d\}$.  This is done unless we are the case $\overline k = r\overline d$, $r \in \{1,2\}$.  But then,  $-\overline r\overline d^2$ is a square, i.e. $-\overline r$ is a square.  This is a contradiction for $\overline r=1$, so we have $\overline k = 2\overline d$, $g(z) = z^2 - 2\overline d^2$, and $g(0) = -2\overline d^2$ is a square.  But then, $g(5 \cdot 3^{-1} \overline d) = -2( 5^2/3^2- 1 )\overline d^2 = -2 (4/3)^2 \overline d^2$ is a square and so we can take $\overline a = 5 \cdot 3^{-1} \overline d$ (this uses that $\gcd(m,30) = 1$).

Thus, for the pair $(\overline{a},\overline{b})$ satisfying conditions (i)-(iv) above, we can find a circle $C_a$ with prime curvature $a\equiv \overline{a} \pmod m$ and then find a circle $C_b$ that is in the pinch family between $C_a$ and $C_d$ and satisfies $b\equiv \overline{b}\pmod m$. With these choices of $a$ and $b$ and either choice of $c$ that completes a Descartes quadruple, we have the desired quadruple. 
\end{proof}

Combining Lemma \ref{lem:special quad!} with Proposition \ref{prop:2-step walk} gives a way to construct core geodesics in a thickened prime component using Theorem~\ref{thm:maintheorempaper2}\bs{(2)} in place of Conjecture \ref{conj:bunyakovsky}:

\begin{proposition}\label{prop:mod7geo}
 Let $m\in\ZZsub{>1}$ be coprime to $30$, and suppose  
\[
\mathcal{W}:=W_{43}(s_0)W_{32}(r_0)
\] has been constructed following Proposition \ref{prop:2-step walk}. Additionally, suppose that $C\equiv 0\pmod m$ in the third row of $\mathcal{W}$. Then, for any circle $C_d$ of prime curvature and for any residue class $\ell \pmod m$, there is a Descartes quadruple $\mathbf{v}$ containing $C_d$ such that  the $\mathcal{W}$-family of $\mathbf{v}$ contains a circle of curvature $\ell\pmod m$. In particular, every geodesic in this $\mathcal{W}$-family is a core geodesic of length at most 2.

\end{proposition}

\begin{proof}
This follows immediately by applying Proposition \ref{prop:2-step walk} with the starting quadruple obtained from Lemma \ref{lem:special quad!} with $k=A^{-1}(\ell-Dd)$, where $A$ and $D$ are from the third row of $\mathcal{W}$. By design, we can take $t=0$ in the construction of 
$\mathcal{W}(t)=W_{43}(s_0)W_{32}(r_0)W_{21}(t)$,
and so the proposition follows. 
\end{proof}

\begin{remark}
   The condition $C\equiv 0\pmod m$ in Proposition \ref{prop:mod7geo} is fundamental to this modified approach. Specifically, we rely on the freedom to choose the residue class of $c$ after we have fixed the residue classes of $a$ and $b$. If $C\not\equiv 0\pmod m$, we lose this freedom and revert back to requiring Conjecture \ref{conj:bunyakovsky} to find prime curvatures in a fixed pinch family.
\end{remark}

\subsection{Explicit examples} We now illustrate a few explicit $\mathcal{W}$-geodesics in the People's packing (Figure~\ref{fig:peoplepacking}) for different choices of moduli $m$. The first example below gives a straightforward application of Proposition \ref{prop:2-step walk} for the People's packing with different choices of moduli $m$. The second example shows that the choice $m=7$ satisfies the hypotheses of Proposition \ref{prop:mod7geo}; we use this method to verify computationally that the thickened prime component of $C_{23}$ in the People's packing contains all residue classes modulo $7$. 

\begin{figure}\label{fig:ex1}
    \centering
    \includegraphics[width=3in]{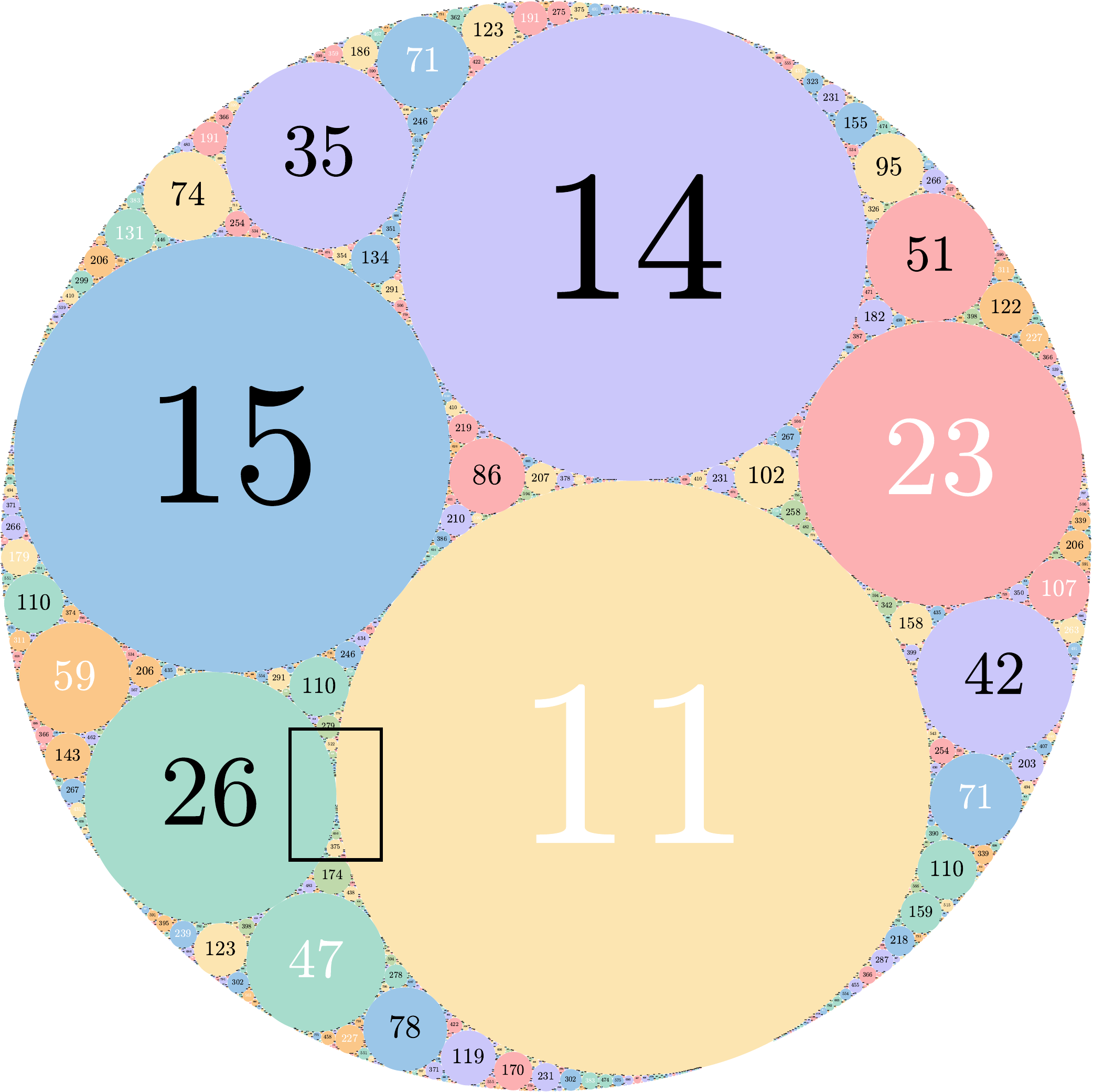} \\
     \includegraphics[width=2in]{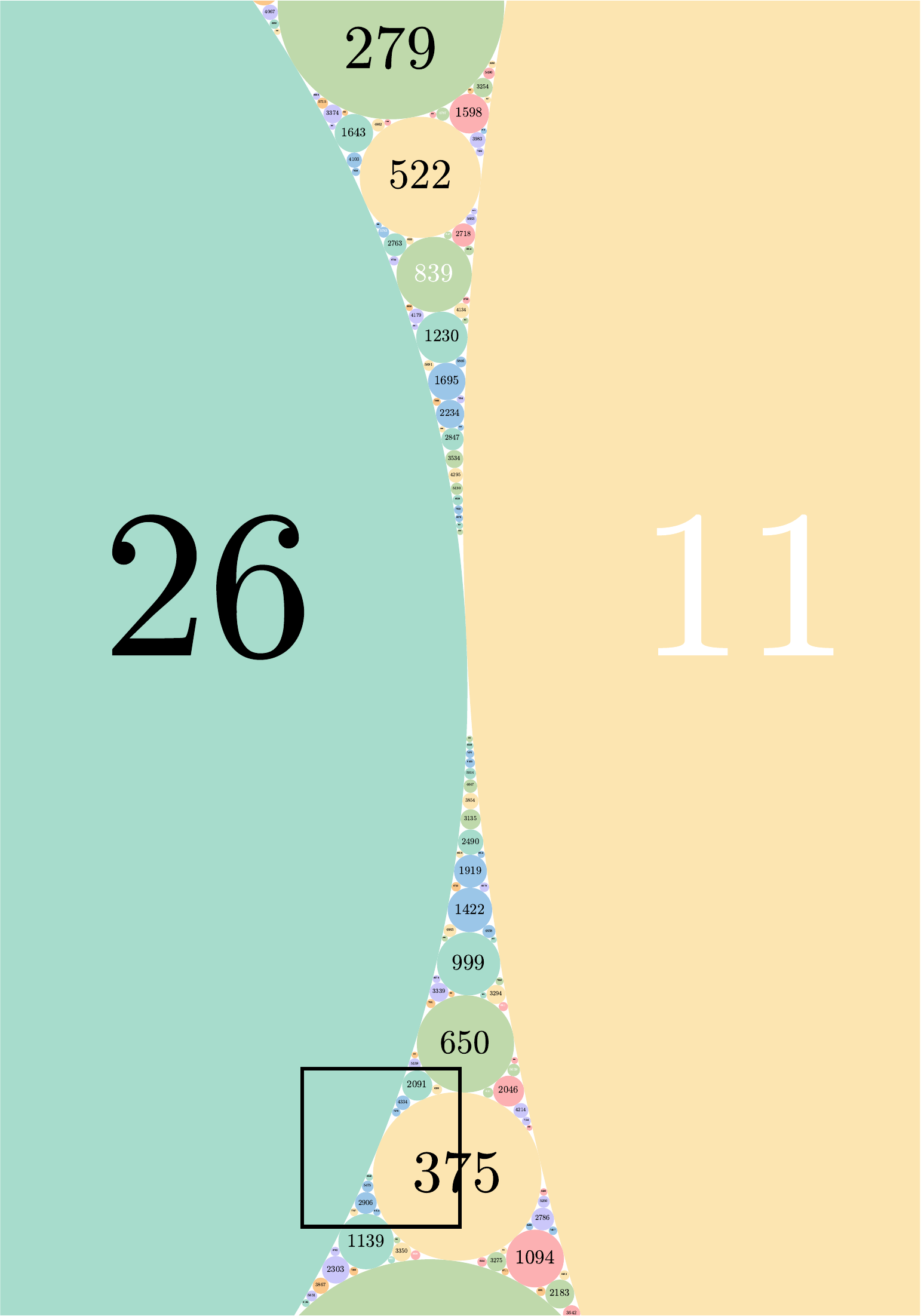}
      \includegraphics[width=2.857in]{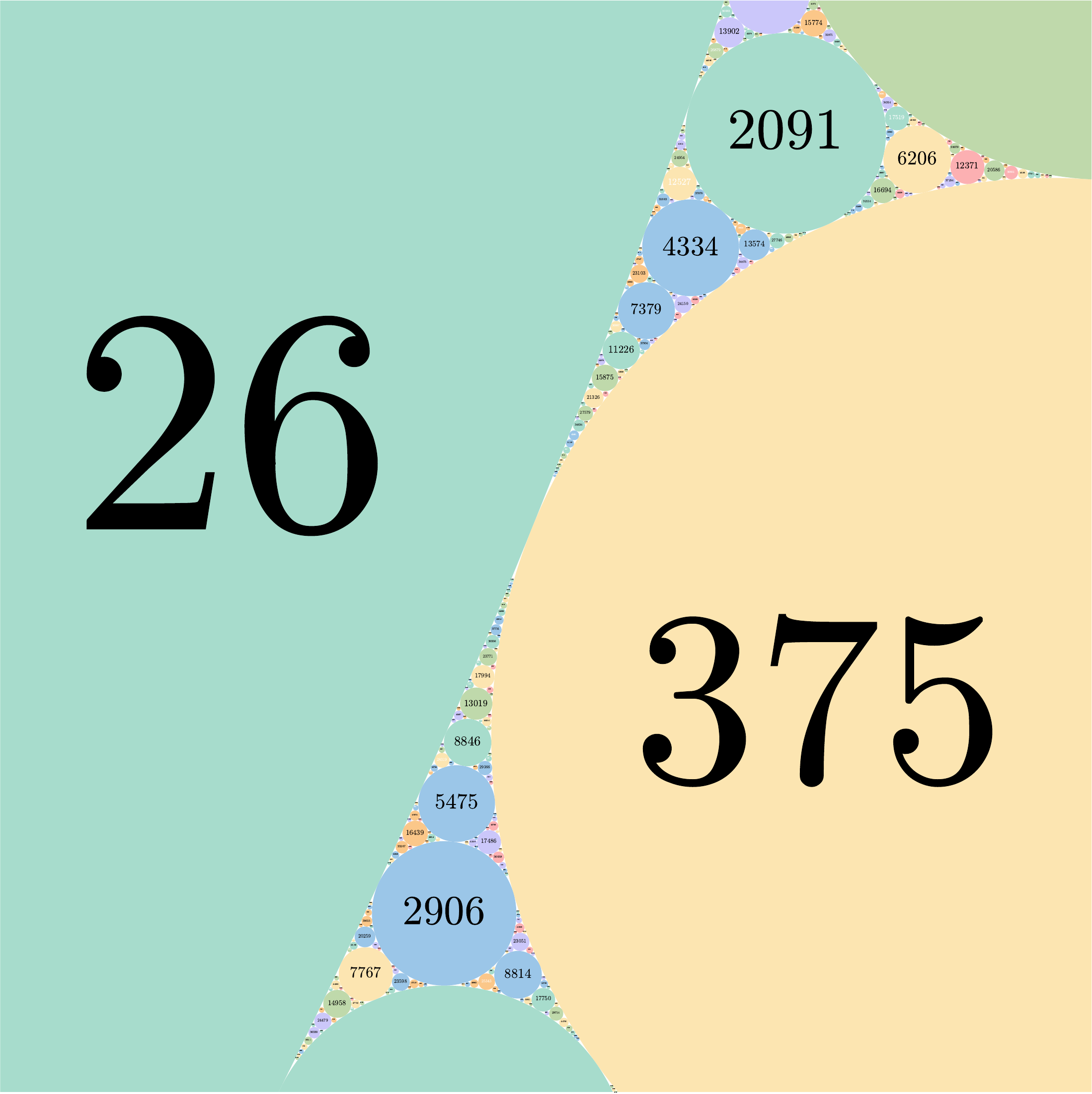}
    \caption{Three successive zooms following the matrix word $W_{43}(5)W_{32}(3)W_{21}(1)$ applied to the quadruple $(14,15,-6,11)$, illustrating the first row in the table of Example~\ref{ex:1}.  Colours indicate residue class modulo $7$.  Font colour distinguishes primes and composites.}
    \label{fig:ex1}
\end{figure}

\begin{example}\label{ex:1} For each choice of modulus $m\in\ZZsub{\geq 7}$ below, we compute the odd integers $s_0$ and $r_0$ that are used in the Proposition \ref{prop:2-step walk} to construct a row of the form \[(A,-A,C,D)\pmod m\] in the word $W_{43}(s_0)W_{32}(r_0)$. Note that this computation depends only on the choice of modulus $m$ (and not on the variable $t$ in the Proposition). Then, for the Descartes quadruple $\mathbf{v}=(14,15,-6,11)$, which generates the People's packing, we compute the linear polynomial in Eq.~(\ref{eq:res class}) that represents residual curvatures appearing in the $\mathcal{W}_{s_0,r_0}(t)$-family of $\mathbf{v}$. 

\vspace{.1cm}
\renewcommand{\arraystretch}{1.2}
\begin{center}
\begin{tabular}{ ||C{1cm}|C{1cm}|C{1cm}|C{2.75cm}|C{3.1cm}|| } 
\hline
 $m$ & $s_0$ & $r_0$& $(A,-A,C,D)$& $\mathcal{W}_{s_0,r_0}(t)$-family \\\hline\hline
 7 & 5 &3 &$(3,4,0,0)$&$-2t - 3$\\
 11 &  3 &5 &$(4, 7, 2, 7)$&$4t - 5$\\
  13 &  19 &19 &$(1, 12, 11, 5)$&$3t + 1$\\
 17 &   11 &25 &$(10, 7, 12, 9)$&$2t$\\
 19&35&9&(2, 17, 11, 13)&$-t - 1$\\
 23 &   33 &11 &$(6, 17, 14, 22)$&$9t - 9$\\
 29 &  53 &43 &$(26, 3, 22, 2)$&$t + 9$\\

 \hline
\end{tabular}
\end{center}
\vspace{.1cm}

As $t$ varies over $2\ZZ$, the $\mathcal{W}_{s_0,r_0}(t)$-families of $\mathbf{v}$ will contain geodesics tying each residue class modulo $m$ to $C_d$. When we choose $t$ so that the first coordinate of $W_{21}(t)\cdot\mathbf{v}^T$ is prime, we can construct core geodesics in those $\mathcal{W}_{s_0,r_0}(t)$-families of $\mathbf{v}$.
\end{example}

\begin{example} Working modulo $m=7$, we illustrate how Lemma \ref{lem:special quad!} can be used in the People's packing to construct length-two core geodesics without a dependence on Conjecture \ref{conj:bunyakovsky} or Theorem~\ref{thm:maintheorempaper2}\bs{(2)}. Recall from Example \ref{ex:1} that for $m=7$, the construction in Proposition \ref{prop:2-step walk} gives $s_0=5$ and $r_0=3$, leading the special row $(3,4,0,0)\pmod 7$ in
\[
\mathcal{W}=W_{43}(5)W_{32}(3).
\]
Let $C_{23}$ be the circle of curvature 23 in the People's packing. In the table below, for each residue class $\ell\pmod 7$, we give a Descartes quadruple $\mathbf{v}=(a,b,c,23)$ satisfying the conditions in Lemma \ref{lem:special quad!} so that by Proposition \ref{prop:mod7geo}, the third coordinate of $\mathcal{W}\cdot \mathbf{v}^T$ is congruent to $\ell\pmod m$. 
\vspace{.1cm}
\renewcommand{\arraystretch}{1.2}
\begin{center}
\begin{tabular}{ ||C{2cm}|C{2cm}|C{2cm}|C{3.65cm}|| } 
\hline
 $\ell\pmod m$ & $a\pmod m$ & $b\pmod m$& starting quadruple\\\hline\hline
0&4&4&$(11,102,58,23)$\\\hline
1&4&6&$(11,258,102,23)$\\\hline
2&4&1&$(11,1562,1134,23)$\\\hline
3&2&1&$(107,918,350,23)$\\\hline
4&2&3&$(107,206,678,23)$\\\hline
5&4&0&$(11,42,-6,23)$\\\hline
6&2&0&$(107,350,42,23)$\\\hline

\hline
\end{tabular}
\end{center}
\vspace{.1cm}

By Lemma \ref{lem:primeclassescircles}, there are no circles of curvature $\ell\equiv 5\pmod 7$ tangent to $C_{23}$; Proposition \ref{prop:mod7geo} applied with the starting quadruple in the above table gives an explicit length-two core geodesic path 
\[C_{23}\longrightarrow C_{11}\longrightarrow C_{7698}.\]
Moreover, the union of circles in the set of quadruples $\mathcal{W}\cdot \mathbf{v}^T$, where $\mathbf{v}$ runs through the starting quadruples given above, contains all residue classes modulo $7$. Thus, the thickened prime component of $C_{23}$ contains all residue classes modulo $7$.

\end{example}

\section{A lower bound for integers represented by thickened prime components}
\label{sec:towards}

In this section, we find a lower bound on the number of distinct curvatures in a thickened prime component by considering circles within two tangencies of the root.  We prove the following:

\begin{theorem}\label{thm:primesdensity}
Let $\mathcal P^{\textrm{th}}$ be a thickened prime component in a primitive integral Apollonian packing.  Let $\kappa^{(2)}(\mathcal P^{\textrm{th}},X)$ be the set of of distinct integers less than or equal to $X$ appearing as curvatures of a circle at most two tangencies away from the root of $\mathcal P^{\textrm{th}}$ (so the circle must either be tangent to the root or tangent to a circle tangent to the root). Then
$$|\kappa^{(2)}(\mathcal P^{\textrm{th}},X)|\gg \frac{X}{(\log \log X)^{1/2}}.$$
\end{theorem}

As a consequence, we immediately obtain Theorem~\ref{thm:primesdensitycorollary}.

\subsection{Outline:  curvatures in two layers}

Let $C_\omega$ be the root of $\mathcal P^{\textrm{th}}$ with curvature $\omega$.  As discussed in Section \ref{quadfam}, the curvatures of the circles tangent to it are given by a translated quadratic form $f_{C_\omega}(x,y) - \omega$, which we will more simply denote $f_\omega(x,y)-\omega$. We denote the set of odd prime such values by 
\[
 \Bdens := \{ \alpha \in \ZZ : \alpha \mbox{ is an odd prime and primitively represented by } f_{C_\omega}(x,y) - \omega \}.
 \]
These are the odd prime curvatures at the first `layer' surrounding $C_\omega$.
 
We will now consider the corresponding collection of circles and their respective translated quadratic forms.  To each $\alpha \in \Bdens$, there is at least one positive definite binary quadratic form $f_\alpha(x,y)$ such that $f_\alpha(x,y) - \alpha$ represents the curvatures of all circles tangent to some circle of curvature $\alpha$ which is itself tangent to $C_\omega$. 
 For $\alpha$ which occur multiple times as values of $f_{C_\omega}(x,y)-\omega$, there may be several forms from which to choose; we choose one arbitrarily.  We refer to $f_\alpha(x,y) - \alpha$ as a \emph{translated curvature form} and define
\[
S_\alpha = \{ n \in \ZZ ,\ n\leq X: n \mbox{ is primitively represented by }f_\alpha(x,y) - \alpha \}.
\]
See Figure~\ref{fig:2levels} for an illustration of these two `layers' in the approach.

\begin{figure}
    \centering
    \includegraphics[width=3in]{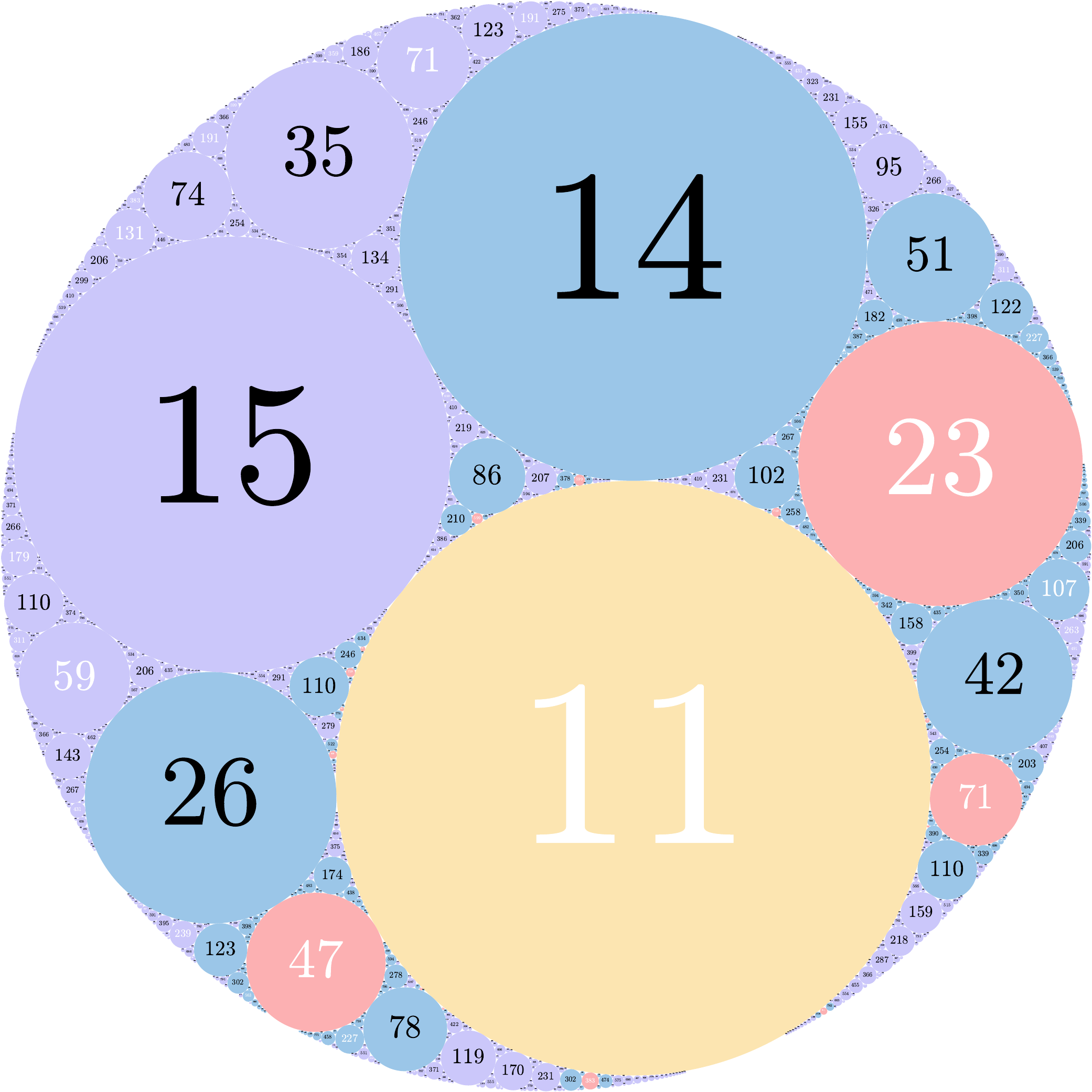} 
    \caption{An illustration of the approach in this section.  In the People's packing (shown up to curvature 3000), the initial circle $C_\omega$ is of curvature 11, in yellow.  The circles of prime curvature tangent to it (the set $\mathcal{B}$) are in red (this is the first `layer').  The curvatures $S_\alpha$ of circles tangent to red circles are shown in blue (this is the second `layer').  Taken together, these two layers form a subset of the thickened prime component.}
    \label{fig:2levels}
\end{figure}

The sets $S_\alpha$, taken together, give a subset of the curvatures that are 
at most two tangencies away from $C_\omega$ in the thickened prime component:
\[
\bigcup_{\alpha \in \Bdens} S_\alpha \subseteq \kappa^{(2)}(\mathcal{P}^{\textrm{th}},X).
\]
 The overall strategy is inclusion-exclusion:
  \[
    \left| \kappa^{(2)}(\mathcal{P}^{\textrm{th}},X) \right| \ge  \left| \bigcup_{ \alpha \in \Bdens_X } S_\alpha \right|
      \ge  \sum_{ \alpha \in \Bdens_X } |S_\alpha| - 
      \sum_{\alpha_1 \neq \alpha_2 \in \Bdens_X } |S_{\alpha_1} \cap S_{\alpha_2} |,
   \]
 where we consider a subset $\Bdens_X \subseteq \Bdens$ growing with a parameter $X$.
   We roughly follow the proof of positive density of curvatures in Apollonian packings in \cite{BourgainFuchs}.

\subsection{Dyadic subdivision}

  The set $\Bdens_X$ is defined as follows
\begin{equation*}
\Bdens_X = \bigcup_{\substack{2\log\log X \\ \le k \le \\ 3\log \log X}} \Bdens^{(k)}, \mbox{ where }
  \Bdens^{(k)} = \Bdens \cap [ 2^k , 2^{k+1} ). \\ 
\end{equation*}

By Theorem~\ref{thm:maintheorempaper2}\bs{(1)} on the representation of primes by shifted quadratic forms, we obtain
\begin{equation} \label{eqn:eta2}
  | \Bdens \cap [ 2^k, 2^{k+1} ) | \gg \frac{2^k}{k^{3/2}}.
  \end{equation}

The idea will now be to consider the values represented by the translated curvature forms associated to elements of the set in Eq.~(\ref{eqn:eta2}). Using an inclusion-exclusion principle, we will obtain lower bounds on the number of values represented by each such form and sum over $\Bdens_X$, and then obtain upper bounds on the number of values represented simultaneously by two such forms. 

\subsection{Lemmata of Bourgain-Fuchs}

We will need a few results from \cite{BourgainFuchs}.

\begin{lemma}
\label{lemma:BG}
Let $f$ be an integral binary quadratic form of discriminant $D = -4a^2$, and suppose that $(\log X)^2 \le a \le (\log X)^3$.  Then the number of distinct values $\le X$ primitively represented by $f$ is $\gg X/a$.
\end{lemma}

\begin{proof}
This follows as in equation (3.13) of \cite{BourgainFuchs}, using results of Blomer-Granville \cite{BlomerGranville}. For a detailed discussion on this lower bound see equation (6) in \cite{Toma}.
\end{proof}

\begin{lemma}[Equation (3.27) of \cite{BourgainFuchs}]
\label{lemma:327}
Let $\omega(n)$ denote the number of distinct prime factors of $n$.  Let $(\log X)^2 \le \alpha_1, \alpha_2 \le (\log X)^3$.  Then for $\alpha_1 \neq \alpha_2$,
\[
  |S_{\alpha_1} \cap S_{\alpha_2}| 
 \ll X 
  \frac{2^{\omega\left( \gcd(\alpha_1,\alpha_2) \right)}}{\alpha_1\alpha_2} 
 \prod_{\substack{p \mid \alpha_1\alpha_2(\alpha_1-\alpha_2) \\ p \text{ prime} \\ p \nmid \gcd(\alpha_1,\alpha_2)}} \left( 1 + \frac{1}{p} \right).
 \]
\end{lemma}

\subsection{Lower bound on $\sum |S_\alpha|$}

With the notation as above, we have the following.

\begin{proposition}
  \label{prop:lowerbound}
  \[
    \sum_{\alpha\in \Bdens_X} |S_\alpha| \gg  \frac{X}{ (\log \log X)^{1/2} }.
\]
\end{proposition}

\begin{proof}
From Lemma~\ref{lemma:BG}, the number of distinct values less than $X$ primitively represented by $f_\alpha$ for a fixed value $\alpha$ is $\gg \frac{X}{\alpha}$.  Therefore we have, using \eqref{eqn:eta2},
\begin{align*}
  \sum_{\alpha\in \Bdens_X} |S_\alpha| &\gg \sum_{\alpha \in \Bdens_X} \frac{X}{\alpha} 
  =  X \sum_{\substack{2 \log \log X \\ \le k \le \\ 3\log\log X}} \sum_{\alpha \in \Bdens^{(k)}} \frac{1}{\alpha} 
  \ge  X \sum_{\substack{2 \log \log X \\ \le k \le \\ 3\log\log X}} \sum_{\alpha \in \Bdens^{(k)}} \frac{1}{2^{k+1}} \\
  \gg&  X \sum_{\substack{2 \log \log X \\ \le k \le \\ 3\log\log X}} \frac{ 2^k}{k^{3/2}} \frac{1}{2^k} 
  =  X \sum_{\substack{2 \log \log X \\ \le k \le \\ 3\log\log X}} \frac{1}{k^{3/2}} 
  \gg  X \frac{ \log\log X}{ (3\log \log X)^{3/2} } \\
  \gg&  \frac{X}{(\log\log X)^{1/2}}.
\end{align*}
\end{proof}

\subsection{Upper bound on $\sum |S_{\alpha_1} \cap S_{\alpha_2}|$}

We now proceed to compute an upper bound on the number of values simultaneously represented by two curvature forms associated to two values in the set $\Bdens_X$.

\begin{proposition}
  \label{prop:upperbound}
  \[
    \sum_{\alpha_1 \neq \alpha_2 \in \Bdens_X} |S_{\alpha_1} \cap S_{\alpha_2}| \ll \frac{X \log\log \log X}{\log\log X}.
  \]
\end{proposition}

We begin by proving the following lemma.

\begin{lemma}
  \label{lemma:nocong2}
  \[
    \sum_{\alpha \in \Bdens_X} \frac{1}{\alpha} \ll  \frac{1}{(\log \log X)^{1/2}}.
  \]
\end{lemma}

\begin{proof}
 First note that
\begin{align*}
    \sum_{\substack{\alpha \in \Bdens_X}} \frac{1}{\alpha}  
    \le
    \sum_{\substack{2\log\log X \\ \le k \le \\ 3\log\log X}} \frac{1}{2^k} \sum_{\alpha \in \Bdens^{(k)}} 1 \ll  \sum_{\substack{2\log\log X \\ \le k \le \\ 3\log\log X}} \frac{1}{2^k} \frac{2^k}{k^{3/2}} \ll (\log\log X)^{-1/2}
\end{align*}
\end{proof}

\color{black}

By Lemma~\ref{lemma:327} (dropping the condition that $p \nmid \gcd(\alpha_1,\alpha_2)$),
\[
  |S_{\alpha_1} \cap S_{\alpha_2}| 
 \ll X 
  \frac{2^{\omega\left( \gcd(\alpha_1,\alpha_2) \right)}}{\alpha_1\alpha_2} 
 \prod_{\substack{p \mid \alpha_1\alpha_2(\alpha_1-\alpha_2) \\ p \text{ prime}}} \left( 1 + \frac{1}{p} \right).
 \]
 Therefore,
\begin{align}
 & \sum_{\alpha_1 \neq \alpha_2 \in \Bdens_X} |S_{\alpha_1} \cap S_{\alpha_2}| \nonumber\\
 & \ll X 
 \sum_{\alpha_1 \neq \alpha_2 \in \Bdens_X} 
 \frac{2^{\omega\left( (\alpha_1,\alpha_2) \right)}}{\alpha_1\alpha_2} 
 \prod_{\substack{p \mid \alpha_1\alpha_2(\alpha_1-\alpha_2) \\ p \text{ prime}}} \left( 1 + \frac{1}{p} \right)\label{intersectionsum}
 \end{align}
 In our case we have that $\alpha_1$ and $\alpha_2$ are prime, so we obtain
 \begin{align*}
(\ref{intersectionsum})&=  X \sum_{\alpha_1 \neq \alpha_2 \in \Bdens_X} \frac{1}{\alpha_1\alpha_2}\left( 1 + \frac{1}{\alpha_1} \right)\left( 1 + \frac{1}{\alpha_2} \right) 
\prod_{\substack{p \mid (\alpha_1-\alpha_2) \\ p \text{ prime}}} \left( 1 + \frac{1}{p} \right) \\
&\ll  X \sum_{\alpha_1 \neq \alpha_2 \in \Bdens_X} \frac{1}{\alpha_1\alpha_2}
\sum_{\substack{q \mid \alpha_1-\alpha_2\\ q \text{ squarefree}}} \frac{1}{q} \\
  &=  
  X \sum_{\substack{q \text{ squarefree}\\ q\leq 2(\log X)^3}} \frac{1}{q}
 \sum_{\alpha_2 \in \Bdens_X} \frac{1}{\alpha_2}
 \sum_{\substack{\alpha_1 \equiv \alpha_2 \pmod q\\ \alpha_1 \in \Bdens_X}} \frac{1}{\alpha_1}
 \end{align*}

We split this summation into two ranges depending on the size of $q$. Let $Q$ be a parameter to be chosen later. For small values of $q$ we completely drop the congruence condition modulo $q$  and by Lemma \ref{lemma:nocong2} we find that
\begin{align*}
 \sum_{\substack{q\leq Q\\q \text{ squarefree}}} \frac{1}{q}
 \sum_{\alpha_2 \in \Bdens_X} \frac{1}{\alpha_2}
 \sum_{\substack{\alpha_1 \equiv \alpha_2 \pmod q\\ \alpha_1 \in \Bdens_X}} \frac{1}{\alpha_1} &\ll \sum_{q\leq Q} \frac{1}{q} \left( \sum_{\alpha \in \Bdens_X} \frac{1}{\alpha}\right)^2 \ll (\log Q) (\log \log X)^{-1}
.\end{align*}
For large values of $q$ we forget the condition that $\alpha_1$ is a value of a fixed shifted binary quadratic form and only keep the size restriction on $\alpha_1$, and the property that $\alpha_1$ is in the congruence class $\alpha_2$ modulo $q$. 
We obtain the bound
\begin{align*}
 &\sum_{\substack{Q<q\leq 2(\log X)^3\\q \text{ squarefree}}} \frac{1}{q}
 \sum_{\alpha_2 \in \Bdens_X} \frac{1}{\alpha_2}
 \sum_{\substack{\alpha_1 \equiv \alpha_2 \pmod q\\ \alpha_1 \in \Bdens_X}} \frac{1}{\alpha_1} \\ &\ll \sum_{\substack{Q<q\leq 2(\log X)^3\\q \text{ squarefree}}} \frac{1}{q}
 \sum_{\alpha_2 \in \Bdens_X} \frac{1}{\alpha_2}
 \sum_{\substack{2\log\log X \\ \le k \le \\ 3\log\log X}} \frac{1}{2^k} \sum_{\substack{\alpha_1 \in \Bdens^{(k)}\\ \alpha_1\equiv \alpha_2\pmod q}} 1 \\
 &\ll \sum_{\substack{Q<q\leq 2(\log X)^3\\q \text{ squarefree}}} \frac{1}{q}
 \sum_{\alpha_2 \in \Bdens_X} \frac{1}{\alpha_2}
 \sum_{\substack{2\log\log X \\ \le k \le \\ 3\log\log X}} \frac{1}{2^k} \left(\frac{2^k}{q}+1\right)\\
 & \ll (\log \log X) \sum_{Q<q\leq (\log X)^3} \frac{1}{q^2}
 \sum_{\alpha_2 \in \Bdens_X} \frac{1}{\alpha_2}
+ \sum_{Q<q\leq (\log X)^3} \frac{1}{q} \sum_{\alpha_2 \in \Bdens_X} \frac{1}{\alpha_2} (\log X)^{-2}\\
&\ll (\log \log X)^{1/2} Q^{-1}  + (\log \log X)^{1/2} (\log X)^{-2}
\end{align*}
We conclude that
$$(\ref{intersectionsum}) \ll X((\log Q) (\log \log X)^{-1} +  (\log \log X)^{1/2} Q^{-1}  + (\log \log X)^{1/2} (\log X)^{-2})$$
Choosing $Q=(\log \log X)^2$ we deduce that
$$\sum_{\alpha_1 \neq \alpha_2 \in \Bdens_X} |S_{\alpha_1} \cap S_{\alpha_2}| \ll X \frac{\log\log \log X}{\log\log X} $$

  \subsection{Putting the two together: proof of Theorem~\ref{thm:primesdensity}}
  
  Propositions \ref{prop:lowerbound} and \ref{prop:upperbound} give
  \[
    \sum_{\alpha \in \Bdens_X} |S_\alpha| \gg \frac{X}{(\log \log X)^{1/2}}, \quad 
    \sum_{\alpha_1 \neq \alpha_2 \in \Bdens_X} |S_{\alpha_1} \cap S_{\alpha_2}| \ll \frac{X\log\log\log X}{\log\log X}.
  \]
 For $X$ sufficiently large this gives
  \[
     \left| \bigcup_{ \alpha \in \Bdens_X } S_\alpha \right|
      \ge  \sum_{ \alpha \in \Bdens_X } |S_\alpha| - 
      \sum_{\alpha_1 \neq \alpha_2 \in \Bdens_X } |S_{\alpha_1} \cap S_{\alpha_2} | 
     \gg \frac{X}{(\log \log X)^{1/2} }. \qed
   \]

\section{Experimental data}
\label{sec:data}

Special-purpose software was developed to collect and analyse data on prime components in Apollonian circle packings \cite{GHapollonian, GHapollonianPrime}.  In this section we collect results.

\subsection{Residue classes}\label{sec:data-res}
For all prime components generated from a prime root with largest prime at most $100$ (255 such components across all primitive integral packings), we verified that there are no further congruence obstructions (i.e., besides those known for the packing as a whole) for any modulus up to 200, for both the prime and the thickened components. For all prime components with prime root having largest entry at most 40 (29 such components), we checked up to modulus 1000.

\subsection{Growth rate of a prime component}\label{sec:growthrateprime}

Conjecture~\ref{conj:growth-primecomp} suggests that the average multiplicity of a curvature in a prime component tends to $\infty$. We support this conjecture with data from two prime components. For the largest prime components in the packings generated from $(-2, 3, 6, 7)$ and $(-6, 11, 14, 15)$, we compute all curvatures up to $10^{12}$, and bin the data into $2000$ bins of length $5\times 10^8$. By plotting $\frac{\cirpr}{\pi(X)}$ against $\log(X)$ in Figure~\ref{fig:primecompgrowth}, we observe that this ratio is increasing, possibly to $\infty$.

\begin{figure}[ht]
	\centering
	\begin{subfigure}[t]{.5\textwidth}
		\centering
		\includegraphics[width=.9\linewidth]{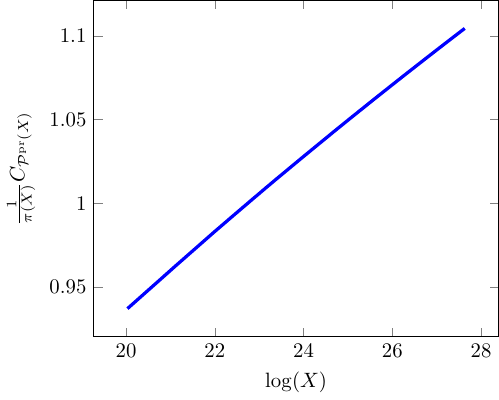}
        \caption{$(-2, 3, 6, 7)$}
	\end{subfigure}%
	\begin{subfigure}[t]{.5\textwidth}
		\centering
		\includegraphics[width=.9\linewidth]{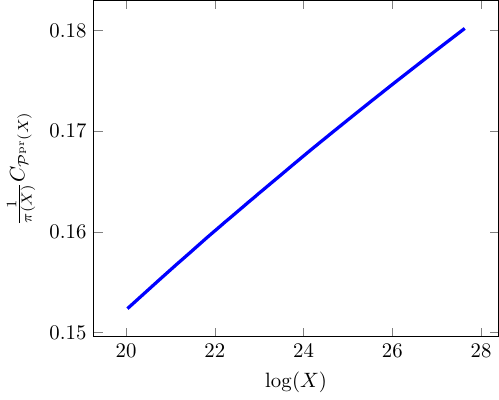}
        \caption{$(-6, 11, 14, 15)$}
	\end{subfigure}
	\caption{Growth rate of the largest prime components in two circle packings.}\label{fig:primecompgrowth}
\end{figure}

We compare to the (somewhat arbitrary) estimate
\[
\cirpr\approx e_{c,\alpha}(X) := \alpha \pi(X) (\log\pi(X))^c
\sim \alpha X (\log X)^{c-1},
\]
for some constants $\alpha$ and $c$.  To estimate the appropriateness of this approximation, we plot $\cirpr/e_{c,1}(X)$ versus $X$ on $\log-\log$ axes in Figure ~\ref{fig:primecompgrowth-asympt}.  For the two packings in question, the graph looks `most' plausible for $c = 0.4649$ and $0.4715$, respectively.

\begin{figure}[ht]
	\centering
	\begin{subfigure}[t]{.5\textwidth}
		\centering
		\includegraphics[width=.9\linewidth]{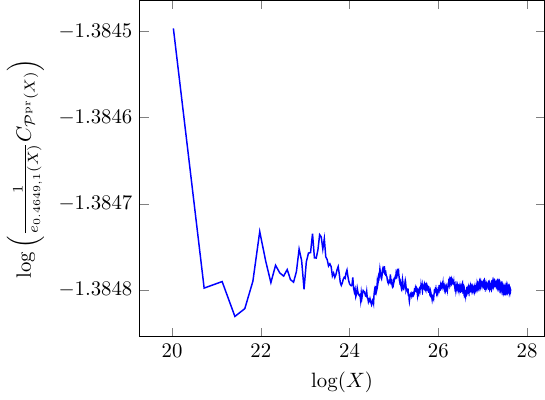}
        \caption{$(-2, 3, 6, 7)$}
	\end{subfigure}%
	\begin{subfigure}[t]{.5\textwidth}
		\centering
		\includegraphics[width=.9\linewidth]{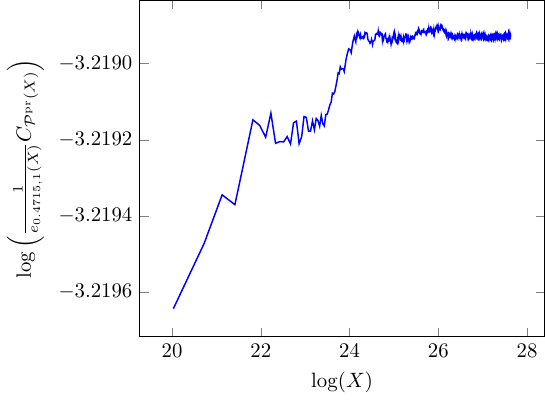}
        \caption{$(-6, 11, 14, 15)$}
	\end{subfigure}
	\caption{Comparison with estimates for the growth rate of the largest prime components in two circle packings.}\label{fig:primecompgrowth-asympt}
\end{figure}

\subsection{Growth rate of a thickened prime component}\label{sec:growthratethick}

Conjecture~\ref{conj:growth-primecomp} suggests that the average multiplicity of a curvature in a thickened prime component also tends to $\infty$.  As in the last section, we consider the largest prime components in the packings generated from $(-2, 3, 6, 7)$ and $(-6, 11, 14, 15)$, we compute all curvatures up to $10^{12}$, and bin the data into $2000$ bins of length $5\times 10^8$. Plotting $\frac{\cirth}{X}$ against $\log(X)$ in Figure~\ref{fig:thickprimecompgrowth}, we again observe that this ratio is increasing, possibly to $\infty$.

\begin{figure}[ht]
	\centering
	\begin{subfigure}[t]{.5\textwidth}
		\centering
		\includegraphics[width=.9\linewidth]{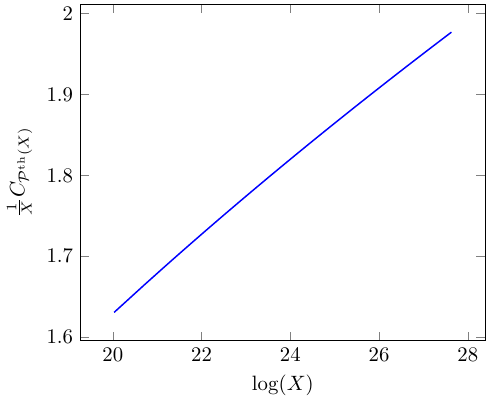}
        \caption{$(-2, 3, 6, 7)$}
	\end{subfigure}%
	\begin{subfigure}[t]{.5\textwidth}
		\centering
		\includegraphics[width=.9\linewidth]{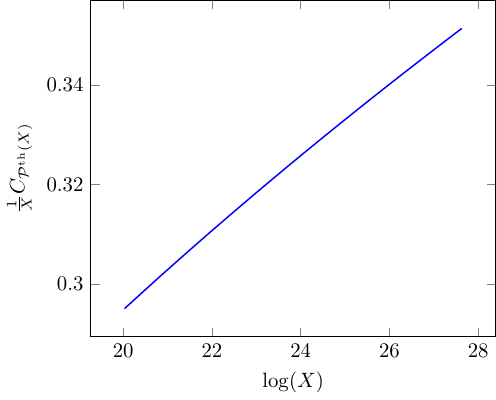}
        \caption{$(-6, 11, 14, 15)$}
	\end{subfigure}
	\caption{Growth rate of the largest thickened prime components in two circle packings.}\label{fig:thickprimecompgrowth}
\end{figure}

In the case of $\cirth$, it is even less clear what potential growth rate to compare to, but we compare again to the same estimate
\[
\cirth\approx e_{c,\alpha}(X) := \alpha \pi(X) (\log\pi(X))^c
\sim \alpha X (\log X)^{c-1},
\]
for some constants $\alpha$ and $c$, where we expect the constants for the prime and thickened component to be related by approximately $c(\Pth) \approx 1 + c(\Ppr)$.  (This might be plausible, for example, if one assumes each circle of curvature $p$ of the prime component contributes around $X/p\log(X)^{1/2}$ circles to the thickening.)  To estimate the appropriateness of this approximation, we plot $\cirth/e_{c,1}(X)$ versus $X$ on $\log-\log$ axes in Figure ~\ref{fig:thickcompgrowth-asympt}.

\begin{figure}[ht]
	\centering
	\begin{subfigure}[t]{.5\textwidth}
		\centering
		\includegraphics[width=.9\linewidth]{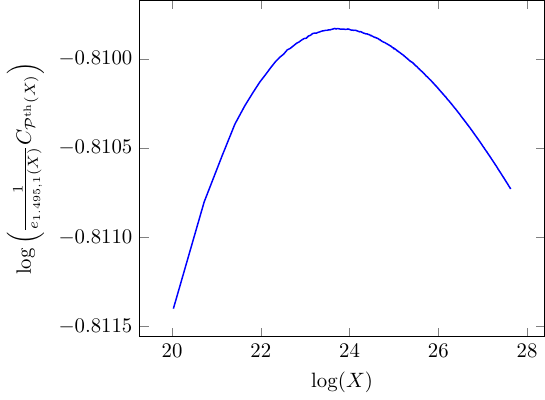}
        \caption{$(-2, 3, 6, 7)$}
	\end{subfigure}%
	\begin{subfigure}[t]{.5\textwidth}
		\centering
		\includegraphics[width=.9\linewidth]{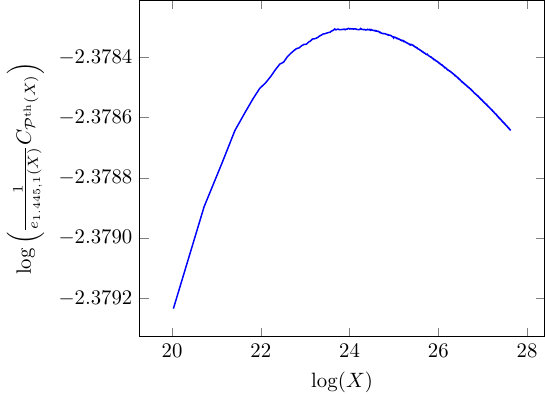}
        \caption{$(-6, 11, 14, 15)$}
	\end{subfigure}
	\caption{Comparison with estimates for the growth rate of the largest thickened prime components in two circle packings.}\label{fig:thickcompgrowth-asympt}
\end{figure}

\subsection{Multiplicities of curvatures}\label{sec:mult}
While the data in Section~\ref{sec:growthratethick} supports the average multiplicity of a curvature in the thickened prime component tending to $\infty$, this does not necessarily imply that the multiplicities of admissible curvatures individually tend to $\infty$. In fact, this statement cannot be true in most packings, as the reciprocity obstructions found in \cite{HKRS} give infinite families of admissible curvatures with multiplicity 0.

Still, we would like to demonstrate that the multiplicities of curvatures grow in a reasonable fashion, supporting Conjecture~\ref{conj:the}.  To this end, we compute all curvatures in a large range for two packings, and give a histogram of the frequencies of the various multiplicities.

In the packing generated by $(-2, 3, 6, 7)$, we compute the multiplicities of all admissible curvatures $c$ for the largest thickened prime component for $c$ in the range $5\cdot 10^{13}+1\leq c\leq 5\cdot 10^{13}+10^9$, and generate Figure~\ref{fig:twoexamples}. This computation took 5.75 days using 64 cores on the Alpine cluster at CU Boulder.

\begin{figure}[ht]
	\centering
	\begin{subfigure}[t]{.5\textwidth}
		\centering
		\includegraphics[width=.9\linewidth]{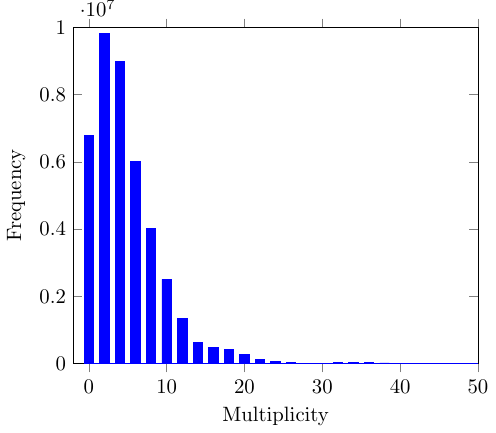}
        \caption{Multiplicity between 0 and 50.}
	\end{subfigure}%
	\begin{subfigure}[t]{.5\textwidth}
		\centering
		\includegraphics[width=.9\linewidth]{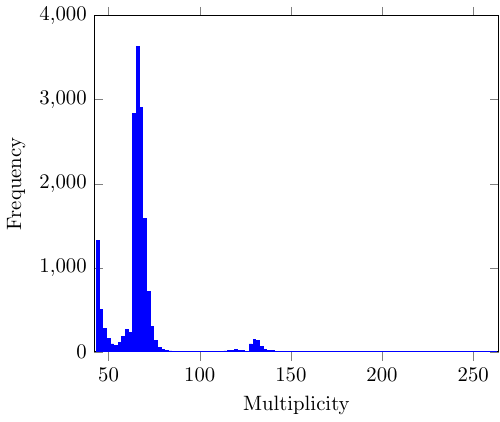}
        \caption{Multiplicity between 44 and 262.}
	\end{subfigure}
	\caption{Multiplicity of curvatures $c\equiv 7\pmod{24}$ for $5\cdot 10^{13}+1\leq c\leq 5\cdot 10^{13}+10^9$ in the largest thickened prime component of $(-2, 3, 6, 7)$. The symmetry of the component causes all multiplicities to be even in the given range.}\label{fig:twoexamples}
\end{figure}

As we choose curvatures $c$ in larger ranges, we expect the average multiplicity to grow, so the ``hump'' should move to the right.  While the range $[5 \cdot 10^{13} + 1, 5 \cdot 10^{13} + 10^9]$ is not large enough to see the main hump ``depart'' from the origin, we can observe that its peak is to the right of the origin, and the average multiplicity has begun to grow. Furthermore, we observe an interesting phenomenon of the right tail: there are multiple further humps, of varying heights (but much smaller than the main one; notice the rescaling of the axes), and unclear genesis.

For example, there are 135 circles of multiplicity between 88 and 126 in this range, yet there are 149 circles of multiplicity 130. Similarly, all multiplicities of 146 and higher occur at most once, save for multiplicity 260, which occurs 3 times. Such a phenomenon does not seem to appear when looking at the entire packing; instead, it appears to be a property of the thickened prime component.

For an example without symmetry, we compute the multiplicities of all admissible curvatures $c$ in the range $10^{14}+1\leq c\leq 10^{14}+10^9$ for the packing generated by $(-6, 11, 14, 15)$. The residue class with the highest multiplicity is $2\pmod{24}$, and the histogram of multiplicities is Figure~\ref{fig:twoexamplesnonsymmetric}.

\begin{figure}[ht]
	\centering
	\begin{subfigure}[t]{.5\textwidth}
		\centering
		\includegraphics[width=.9\linewidth]{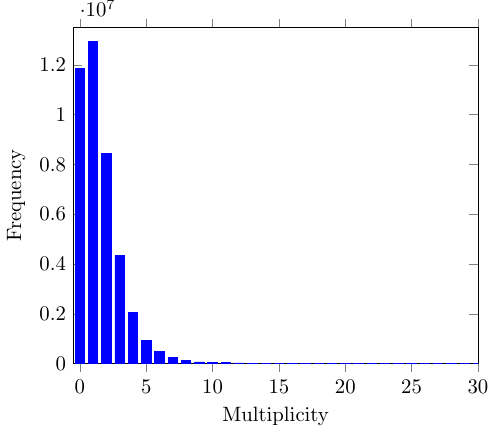}
        \caption{Multiplicity between 0 and 30.}
	\end{subfigure}%
	\begin{subfigure}[t]{.5\textwidth}
		\centering
		\includegraphics[width=.9\linewidth]{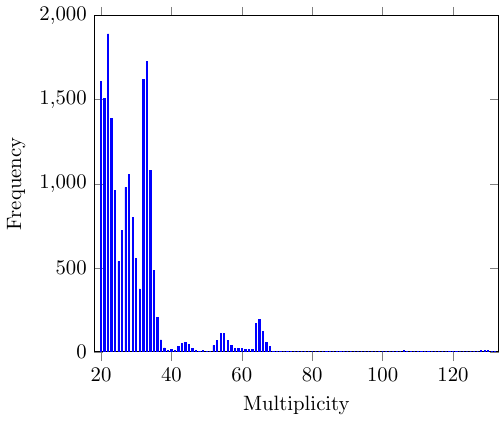}
        \caption{Multiplicity between 18 and 133.}
	\end{subfigure}
	\caption{Multiplicity of curvatures $c\equiv 2\pmod{24}$ for $10^{14}+1\leq c\leq 10^{14}+10^9$ in the largest thickened prime component of $(-6, 11, 14, 15)$.}\label{fig:twoexamplesnonsymmetric}
\end{figure}

We observe properties similar to those of the packing $(-2, 3, 6, 7)$: the average multiplicity has started to grow and form a hump, with a few smaller humps of varying heights.

\subsection{Counting the number of prime root components}

Conjecture \ref{PCRconjecture} states that
\[N^{\text{root}}_\Pfull(X)\sim c\frac{\cirfull}{\log X},\]
where $c=0.9159\ldots$ is an explicit constant. We can collect data, which, at the very least, supports the rough order of magnitude of $\frac{\cirfull}{\log X}$, with a constant plausibly close to $c$.

The difficulty in obtaining good data is likely due to several factors, including
\begin{itemize}
    \item The presence of a secondary order term, which is still moderately large at the ranges we can compute to; very roughly, in order for the second term to be less than $1/N$ of the first term, one needs to consider $X>e^{1.5N}$.
    \item The discrepancy between $\pi(x)$ and $\frac{x}{\log(x)}$ is still quite large for $x\leq 10^{10}$.
\end{itemize}
The size of $X$ required to mitigate these factors sufficiently is likely far beyond any feasibly computable range.

Recalling the heuristic discussion surrounding Conjecture \ref{PCRconjecture}, for any $c''\in\mathbb{R}$ define
\[
f_{c''}(X) = N^{\text{root}}_\Pfull(X) \left/ \middle(c\frac{\cirfull}{\log X} - c''\frac{ \cirfull}{(\log X)^2}\right).
\]
If the main asymptotic is correct, then $\lim_{X\rightarrow\infty} f_{c''}(x)=1$ for any constant $c''$. 

We compute $f_{c''}(X)$ for $X\leq 10^{10}$ in the packings generated by $(-2, 3, 6, 7)$ and $(-6, 11, 14, 15)$ for $c''=0, 2$, with the results shown in Figure~\ref{fig:primerootsasymptotic}.

\begin{figure}[ht]
	\centering
	\begin{subfigure}[t]{.5\textwidth}
		\centering
		\includegraphics[width=.9\linewidth]{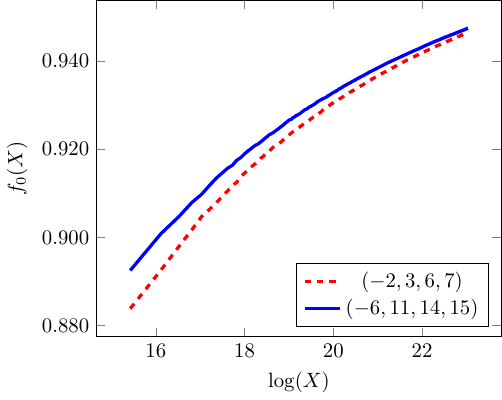}
        \caption{$c''=0$}
	\end{subfigure}%
	\begin{subfigure}[t]{.5\textwidth}
		\centering
		\includegraphics[width=.9\linewidth]{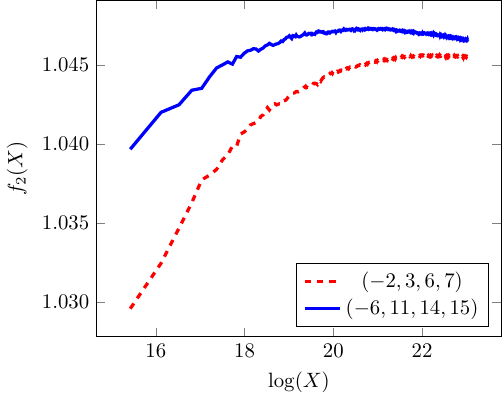}
        \caption{$c''=2$}
	\end{subfigure}
	\caption{Testing the asymptotic for $N^{\text{root}}_\Pfull(X)$ in the packings generated by $(-2, 3, 6, 7)$ and $(-6, 11, 14, 15)$ with $X=i\cdot 10^7$, $1\leq i\leq 10^3$.}\label{fig:primerootsasymptotic}
\end{figure}

When $c''=0$, $f_{0}(X)$ compares the prime root components count to the main term. In both packings, we see that $0.9<f_0(X)<1$, and $f_0(X)$ is increasing, possibly to 1. With the correction term coming from $c''=2$, we are now overcounting the prime root components, although the graph has started to dip and start to potentially approach 1. In any case, the data supports that Conjecture \ref{PCRconjecture} is plausible.

\section{Open questions}
\label{sec:open}

We propose some open questions for further investigation:

\begin{enumerate}
    \item Is it possible to prove a lower bound on the growth of a prime or thickened prime component (the number of circles with curvatures less than $X$), or to prove Conjecture~\ref{conj:growth-primecomp}?
    \item Is it possible to prove bounds approaching the heuristics of Section~\ref{sec:count-prime-comp} for the number of prime components?
    \item Is it possible to obtain positive-density in prime or thickened prime components?  That is, lower bounds on the number of curvatures in the prime component that are $\gg \pi(X)$ (the prime counting function) or $\gg X$ in the case of the thickened prime component?  This is work in progress by a subset of the authors here.
    \item What causes the `humps' observed in Section~\ref{sec:mult}?
    \item Given the set of primes in a fixed prime component, what bounds can one prove on gaps between these primes (e.g. can one prove that there are infinitely many pairs of primes that are $1,000,000$ apart)? Perhaps Maynard's tools \cite{Maynard} can be modified to apply in this case.
    \item Do these results hold in other Apollonian-like packings such as those of \cite{FSZ,KapovichKontorovich,NakamuraKontorovich,StangeVis}?
    \item What about $r$-almost prime components?  (We consider the count of curvatures in such components in forthcoming work.)
    \item What about thickening twice (by which we mean augmenting a prime component by all circles within two tangencies of the component), or more times?  Do multiple thickenings put stronger tools within reach?
\end{enumerate}

\bibliographystyle{alpha}
\bibliography{refs}

\end{document}